\documentclass{amsart}
\usepackage{graphicx}
\usepackage{hyperref,bm,caption,amsbsy,enumerate,amsmath,amsthm,
amssymb,mathtools,amsfonts,multirow,verbatim,tikz,tikz-cd,thmtools,thm-restate}
\usepackage[alphabetic]{amsrefs}
\usetikzlibrary{matrix,arrows,decorations.pathmorphing}
\usepackage[top=1.25in, bottom=1in, left=1in, right=1in]{geometry}
\usepackage{url}
\DefineSimpleKey{bib}{myurl}

\usepackage{chngcntr}
\counterwithin{figure}{section}

\newcommand\myurl[1]{\url{#1}}

\BibSpec{webpage}{%
  +{}{\PrintAuthors} {author}
  +{,}{ \textit} {title}
  +{}{ \parenthesize} {date}
  +{,}{ \myurl} {myurl}
  +{,}{ } {note}
  +{.}{ } {transition}
}

\newtheorem{innercustomthm}{Theorem}
\newenvironment{customthm}[1]
  {\renewcommand\theinnercustomthm{#1}\innercustomthm}
  {\endinnercustomthm}

\newtheorem{innercustomcor}{Corollary}
\newenvironment{customcor}[1]
  {\renewcommand\theinnercustomcor{#1}\innercustomcor}
  {\endinnercustomcor}

\newtheorem{thm}{Theorem}[section]
\newtheorem{prop}[thm]{Proposition}
\newtheorem{define}[thm]{Definition}

\newtheorem{lem}[thm]{Lemma}

\newtheorem{rem}[thm]{Remark}

\newcommand{\ve}[1]{\boldsymbol{\mathbf{#1}}}
\newcommand{\R}{\mathbb{R}}

\newcommand{\Z}{\mathbb{Z}}

\renewcommand{\d}{\partial}
\renewcommand{\subset}{\subseteq}

\renewcommand{\tilde}{\widetilde}
\renewcommand{\bar}{\overline}
\renewcommand{\hat}{\widehat}
\newcommand{\iso}{\cong}

\DeclareMathOperator{\id}{{id}}

\DeclareMathOperator{\im}{{im}}

\DeclareMathOperator{\Mod}{{mod}}

\DeclareMathOperator{\PSL}{{PSL}}

\DeclareMathOperator{\Spin}{{Spin}}

\DeclareMathOperator{\Sym}{{Sym}}

\DeclareMathOperator{\Tors}{{Tors}}

\newcommand{\bD}{\mathbb{D}}

\newcommand{\bL}{\mathbb{L}}

\newcommand{\bT}{\mathbb{T}}

\newcommand{\cA}{\mathcal{A}}

\newcommand{\cC}{\mathcal{C}}
\newcommand{\cD}{\mathcal{D}}

\newcommand{\cF}{\mathcal{F}}

\newcommand{\cH}{\mathcal{H}}

\newcommand{\cJ}{\mathcal{J}}

\newcommand{\cL}{\mathcal{L}}
\newcommand{\cM}{\mathcal{M}}
\newcommand{\cN}{\mathcal{N}}

\newcommand{\cP}{\mathcal{P}}

\newcommand{\cT}{\mathcal{T}}

\newcommand{\frP}{\mathfrak{P}}

\newcommand{\frj}{\mathfrak{j}}

\newcommand{\frs}{\mathfrak{s}}
\newcommand{\frt}{\mathfrak{t}}

\newcommand{\1}{\mathbf{1}}

\title{Quasi-stabilization and basepoint moving maps in link Floer homology}
\author{Ian Zemke}
\address{Department of Mathematics\\ UCLA\\ 520 Portola Plaza, Los Angeles, CA 90025, USA}
\email{ianzemke@math.ucla.edu}

\begin{document}

\begin{abstract}We analyze the effect of adding, removing, and moving basepoints on link Floer homology. We prove that adding or removing basepoints via a procedure called quasi-stabilization is a natural operation on a certain version of link Floer homology, which we call $CFL_{UV}^\infty$. We consider the effect on the full link Floer complex of moving basepoints,  and develop a simple calculus for moving basepoints on the link Floer complexes. We apply it to compute the effect of several diffeomorphisms corresponding to moving basepoints. Using these techniques we prove a conjecture of Sarkar about the map on the full link Floer complex induced by a finger move along a link component.
\end{abstract}
\maketitle

\tableofcontents
\section{Introduction} Introduced by Ozsv\'{a}th and Szab\'{o}, Heegaard Floer homology associates algebraic invariants to closed three manifolds. To a three manifold $Y$ with embedded nullhomologous knot $K$, there is a refinement of Heegaard Floer homology called knot Floer homology, introduced by Ozsv\'{a}th and Szab\'{o} in \cite{OSKnots} and independently by Rasmussen in \cite{RasmussenCFr}. A similar invariant was defined by Ozsv\'{a}th and Szab\'{o} for links in \cite{OSMulti}.

To a nullhomologous knot $K\subset Y$ with two basepoints $z$ and $w$ and a relative $\Spin^c$ structure $\frt\in \Spin^c(Y,K)$, in \cite{OSKnots}, Ozsv\'{a}th and Szab\'{o} define  a $\Z\oplus \Z$-filtered chain complex $CFK^\infty(Y,K,w,z,\frt)$. The $\Z\oplus \Z$-filtered chain homotopy type of $CFK^\infty(Y,K,w,z,\frt)$ is an invariant of the data $(Y,K,w,z,\frt)$.

One of the nuances of Heegaard Floer homology is the dependence on basepoints. In the case of closed three manifolds, if $\ve{w}\subset Y$ is a collection of basepoints, $w\in \ve{w}$ and $\gamma$ is a curve in $\pi_1(Y,w)$, then one can consider the diffeomorphism $\phi_\gamma$ resulting from a finger move along $\gamma$. According to \cite{JTNaturality}, the based mapping class group $MCG(Y,w)$ acts on $CF^\circ(Y,\ve{w},\frs)$ and hence there is an induced map $(\phi_\gamma)_*$ on the closed three manifold invariants, $CF^\circ(Y,\ve{w},\frs)$, which is a $\Z_2[U]$-equivariant chain homotopy type. In \cite{HFplusTQFT}, the author computes the equivariant chain homotopy type of $(\phi_\gamma)_*$ to be
\[(\phi_\gamma)_*\simeq 1+[\gamma]\Phi_w,\] where $[\gamma]$ is the $\Lambda^* (H_1(Y;\Z)/\Tors)$ action and $\Phi_w$ is a map which we will describe below.

In this paper, we consider the analogous question about basepoint dependence for link Floer homology. In link Floer homology, the basepoints are constrained to be on the link component, so the analogous operation is to consider the map on link Floer homology induced by a finger move $\varsigma$ around a link component, in the positive direction according to the links orientation. Using grid diagrams, Sarkar computes in \cite{Smovingbasepoints} the map associated to the diffeomorphism $\varsigma$ on a certain version of link Floer homology (the associated graded complex) for links in $S^3$. For links in arbitrary three manifolds, and for the induced map on the full link Floer complex, he only conjectures the formula. We prove his formula in full generality (Theorem \ref{thm:B}), but before we state that theorem we will provide a brief description of the complexes and maps which appear.

We will work with a slightly different version of $CFL^\infty$ than the one which most often appears in the literature. To a multibased link $\bL=(L,\ve{w},\ve{z})$ inside of $Y$, as well as a $\Spin^c$ structure $\frs\in \Spin^c(Y)$, we construct a chain complex \[CFL_{UV}^\infty(Y,\bL,\frs),\] which is a module over the polynomial ring $\Z_2[U_{\ve{w}},V_{\ve{z}}]$, generated by variables $U_w$ with $w\in \ve{w}$ and $V_z$ with $z\in \ve{z}$. The module $CFL_{UV}^\infty(Y,\bL,\frs)$ has generators of the form \[\ve{x}\cdot U_{\ve{w}}^I V_{\ve{z}}^J=\ve{x}\cdot U_{w_1}^{i_1}\cdots U_{w_n}^{i_n} V_{z_1}^{j_1}\cdots V_{z_n}^{j_n},\]  for multi-indices $I=(i_1,\dots, i_n)$ and $J=(j_1,\dots, j_n)$, though we identify two variables $V_z$ and $V_{z'}$ if $z$ and $z'$ are on the same link component. As such, $CFL_{UV}^\infty(Y,\bL,\frs)$ has a filtration by $\Z^{|\ve{w}|}\oplus \Z^{|L|}$  given by filtering over powers of the variables, where $|L|$ denotes the set of components of $L$.

As with the free stabilization maps from \cite{HFplusTQFT}, to define functorial maps corresponding to adding or removing basepoints in link Floer homology, we must work with colored complexes. A coloring $(\sigma,\frP)$ of a link with basepoints, $(L,\ve{w},\ve{z})$, is a set $\frP$ indexing a collection of formal variables, together with a map $\sigma:\ve{w}\cup \ve{z}\to \frP$ which maps all $\ve{z}$-basepoints on a component of $L$ to the same color. Given a coloring $(\sigma, \frP)$ of a link $\bL=(L,\ve{w},\ve{z})$, we create a $\Z_2[U_\frP]$-chain complex 
\[CFL_{UV}^\infty(Y,\bL,\sigma, \frP,\frs).\] The powers of the $U_\frP$ variables yield a filtration by $\Z^{\frP}$, which we call the $\frP$-filtration.

In the context link Floer homology, analogous to adding or removing a free basepoint in a closed 3-manifold, one can add or remove a pair of adjacent basepoints, $w$ and $z$, on a link component. Some authors refer to this operation as a ``special stabilization''. Manolescu and Ozsv\'{a}th consider certain questions about the operation in \cite{MOIntSurg}, calling it ``quasi-stabilization'', which is the phrase we will use. A full description and proof of naturality of the operation has not been completed, so we do that in this paper:

\begin{customthm}{A}\label{thm:A} Suppose $z$ and $w$ are new basepoints on a link $\bL=(L,\ve{w},\ve{z})$, ordered so that $w$ comes after $z$, which aren't separated by any basepoints in $\ve{w}$ or $\ve{z}$. If $\sigma:\ve{w}\cup \ve{z}\to \frP$ is a coloring which is extended by $\sigma':\ve{w}\cup \ve{z}\cup \{w,z\}\to \frP$, then there are $\frP$-filtered $\Z_2[U_{\frP}]$-chain maps
\[S_{w,z}^+:CFL^\infty_{UV}(Y,L,\ve{w},\ve{z},\sigma,\frP,\frs)\to CFL^\infty_{UV}(Y,L,\ve{w}\cup \{w\},\ve{z}\cup \{z\},\sigma',\frP,\frs)\] and 
\[S_{w,z}^-:CFL^\infty_{UV}(Y,L,\ve{w}\cup \{w\},\ve{z}\cup \{z\},\sigma',\frP,\frs)\to CFL^\infty_{UV}(Y,L,\ve{w},\ve{z},\sigma,\frP,\frs),\] which are well defined invariants, up to $\frP$-filtered, $\Z_2[U_\frP]$-equivariant chain homotopy. If $z$ comes after $w$, there are maps $S_{z,w}^+$ and $S_{z,w}^-$ defined analogously.
\end{customthm}

Following \cite{Smovingbasepoints}, we consider endomorphisms $\Phi_w$ and $\Psi_z$ of $CFL_{UV}^\infty(Y,\bL,\frs)$ (denoted $\Phi_{i,j}$ and $\Psi_{i,j}$ in his notation). We can think of the maps $\Phi_w$ and $\Psi_z$ as formal derivatives of the differential $\d$ with respect to the variables $U_w$ and $V_z$ respectively.  The maps $\Phi_w$ and $\Psi_z$ are invariants of $CFL_{UV}^\infty(Y,\bL,\sigma, \frP,\frs)$ up to $\frP$-filtered chain homotopy.

The maps $\Psi_z$ can be thought of as an analog for the relative homology maps defined in \cite{HFplusTQFT} for the closed three manifold invariants, since they play the role in the basepoint moving maps for link Floer homology that the relative homology maps introduced in \cite{HFplusTQFT} played in the basepoint moving maps for the closed three-manifolds invariants. Indeed the objects $CFL_{UV}^\infty(Y,\bL,\frs)$ and the maps $\Psi_z$ and $S_{w,z}^{\pm}$ fit into the framework of a ``graph TQFT'' for surfaces embedded in four-manifolds with some extra decoration, similar to the TQFT for $\hat{HFL}$ constructed using sutured Floer homology by Juh\'{a}sz in \cite{JCob} and considered further in \cite{JMConcordance} for concordances. Such a TQFT construction for $CFL_{UV}^\infty$ will appear in a future paper.

 We finally state Sarkar's conjecture, cast into the framework of $CFL^\infty_{UV}(Y,\bL, \sigma,\frP,\frs)$:

\begin{customthm}{B}\label{thm:B}
 Suppose that $\bL=(L,\ve{w},\ve{z})$ is a multibased link in an arbitrary 3-manifold $Y$ and $K$ is a component of $L$. Suppose that the basepoints on $K$ are $z_1,w_1,\dots, z_n, w_n$. Letting $\varsigma$ denote the diffeomorphism resulting from a finger move around a link component $K$, the induced map $\varsigma_*$ on $CFL_{UV}^\infty(Y,\bL,\sigma,\frP,\frs)$ has the $\frP$-filtered equivariant chain homotopy type
\[\varsigma_*\simeq 1+\Phi_K\Psi_K\] where 
\[\Phi_K= \sum_{j=1}^n \Phi_{w_j}\qquad \text{ and } \qquad \Psi_K=\sum_{j=1}^n \Psi_{z_j}.\]
\end{customthm}

Sarkar's conjecture for the effect on the filtered link Floer complex, which we will denote by $CFL^\infty(Y,\bL,\frt)$ for $\frt$ a relative $\Spin^c(Y,L)$ structure, follows by setting $(\sigma,\frP)$ to be the trivial coloring (i.e. $\frP=\ve{w}\cup |L|$ and $\sigma:(\ve{w}\cup \ve{z})\to (\ve{w}\cup |L|)$ the natural map) since the complex $CFL^\infty(Y,\bL,\frt)$ becomes a $\Z_2$-subcomplex of $CFL^{\infty}_{UV}(Y,\bL,\sigma, \frP,\frs)$ which is preserved by $\varsigma_*$, where $\frs$ is the underlying $\Spin^c$ structure associated to the relative $\Spin^c$ structure $\frt$.

There are several other formulations of the above conjecture for different versions of link Floer homology. For example, the conjectured formula for $\varsigma_*$ on $CFK^\infty(K)$ for $K\subset S^3$ is useful for computations in the involutive Heegaard Floer homology theory developed by Hendricks and Manolescu (see \cite{HMInvolutive}*{Sec. 6}). In their notation, for a choice of diagrams, the complex $CFK^\infty(K)$ for a knot $K\subset S^3$ is generated by elements of the form $[\ve{x},i,j]$ where $i$ and $j$ satisfy $A(\ve{x})=i-j$, and $A$ denotes the Alexander grading. In their notation, the $U$ map takes the form $U\cdot[\ve{x},i,j]=[\ve{x},i-1,j-1]$. Again the complex $CFK^\infty(K)$ is a $\Z_2$-subcomplex
\[CFK^\infty(K)\subset CFL_{UV}^\infty(S^3,K,w,z,\frs_0)\] which is preserved by $\varsigma_*$.  Recasting Theorem \ref{thm:B} into this notation and recalling that we are using coefficients in $\Z_2$, we arrive at the following:

\begin{customcor}{C}\label{cor:C}For a knot $K\subset S^3$, the involution $\varsigma_*$ on $CFK^\infty(S^3,K,w,z)$ takes the form
\[\varsigma_*\simeq 1+U^{-1}\bigg(\sum_{\substack{i,j\ge 0\\ i \text{ odd}}}\d_{ij}\bigg)\circ\bigg(\sum_{\substack{i,j\ge 0\\ j \text{ odd}}}\d_{ij}\bigg),\] if we write the differential $\d=\sum_{i,j\ge 0 } \d_{ij}$. Here $\d_{ij}$ decreases the first filtration by $i$, and the second filtration by $j$.
\end{customcor}

For other flavors, such as $\hat{CFL}$ or $CFL^-$, the formula conjectured by Sarkar also follows, since those cases correspond to setting various variables equal to zero in the formula for $\varsigma_*$.

In addition, we consider the effect of another diffeomorphism which naturally appears in the involutive setting (cf. \cite{HMInvolutive}). Suppose that $K$ is a component of a link $\bL$ and suppose that the basepoints of $K$ are $z_1,w_1,\cdots, z_n$ and $w_n$, appearing in that order. We can consider the diffeomorphism $\tau:(Y,\bL)\to (Y,\bL)$ which twists $(\frac{1}{n})^{\textrm{th}}$ of the way around $K$. The diffeomorphism $\tau$ maps $z_{i}$ and $w_i$ to $z_{i+1}$ and $w_{i+1}$ respectively, with indices taken modulo $n$. If $(\sigma,\frP)$ is a coloring of $\bL$ which sends all of the $\ve{w}$-basepoints on $K$ to the same color, then $\tau$ naturally induces an automorphism of
\[CFL_{UV}^\infty(Y,\bL,\sigma,\frP,\frs).\] Using the techniques of this paper, we can compute the following:

 \begin{customthm}{D}\label{thm:D}Suppose that $\bL$ is an embedded link in $Y$, and $K$ is a component of $\bL$ with basepoints $z_1,w_1,\dots, z_n,$ and $w_n$, appearing in that order. Assume that $n>1$. If $\tau$ denotes the diffeomorphism induced by twisting $(\frac{1}{n})^{\textrm{th}}$ of the way around $K$, then for a coloring where all $\ve{w}$-basepoints on $K$ have the same color, we have
 \[\tau_*\simeq (\Psi_{z_1}\Phi_{w_1}\Psi_{z_2} \Phi_{w_2} \cdots \Phi_{w_{n-1}} \Psi_{z_n} \Phi_{w_n})
 +(\Phi_{w_1}\Psi_{z_2} \Phi_{w_2} \cdots \Phi_{w_{n-1}} \Psi_{z_n}).\]
 \end{customthm}
\subsection{Organization}

In Section \ref{sec:p-filteredcomplexesanddefinitions} we define the complexes which will appear in this paper, as well as their algebraic structures as $\frP$-filtered chain complexes over certain modules. In Section \ref{sec:PhiPsi} we define the maps $\Phi_w$ and $\Psi_z$ which feature prominently in this paper. In Sections \ref{sec:prelimquasistab}-\ref{sec:quasistabilizationnatural} we define quasi-stabilization maps and show that they are independent of the choice of diagrams and auxiliary data, proving Theorem \ref{thm:A}. In Sections \ref{sec:freestabcommute} and \ref{sec:furtherrelations} we prove useful relations amongst the maps $\Psi_z$, $\Phi_w$ and $S_{w,z}^{\pm}$. In Section \ref{sec:basepointmovingmaps} we compute several maps associated with moving basepoints, proving Theorems \ref{thm:B} and \ref{thm:D}.

\subsection{Acknowledgments}

I would like to thank my advisor, Ciprian Manolescu, for helpful conversations, especially about quasi-stabilization. I would also like to thank Faramarz Vafaee and Robert Lipshitz for helpful conversations.

\section{\texorpdfstring{$\frP$}{P}-filtered complexes, and the complexes \texorpdfstring{$CFL^\infty$}{CFL\^{}infty} and \texorpdfstring{$CFL_{UV}^\infty$}{CFL\_UV\^{}infty}}
\label{sec:p-filteredcomplexesanddefinitions}
There are several variations of the full knot or link Floer complex, each with slightly different notation. We now list the complexes and will use, their module and filtration structure. We will define them more precisely, and also give a brief outline of $\frP$-filtered complexes, in the following subsections.

\subsection{Summary of the complexes which will appear}
In the following table, we list several different versions of the full link or knot Floer complex. We define $\bar{U}_w$ to be $U_wV_z$ for any $z$ in the link component containing $w$. Similarly $\Spin^c(Y,L)_\frs$ denotes the fiber over $\frs\in \Spin^c(Y)$ of the natural map $\Spin^c(Y,L)\to \Spin^c(Y)$ mapping a relative $\Spin^c$ structure to the underlying nonrelative $\Spin^c$ structure. We will primarily be interested in $CFL_{UV}^\infty$, since it is the construction which allows for natural maps to be most easily defined.

\begin{center}
    \begin{tabular}{ | l | l |  l   | l|}
    \hline Notation:& Input:& Structure:& Generators\\ \hline
     $CFL_{UV}^{\infty}(Y,\bL,\frs)$& $\frs\in \Spin^c(Y)$& $\Z^{|\ve{w}|}\oplus \Z^{|L|}$-filtered, $\Z_2[U_{\ve{w}}, V_{\ve{z}}]$-module&$U_{\ve{w}}^I V_{\ve{z}}^J\cdot \ve{x}$\\ \hline $CFL^\infty(Y,\bL,\frt)$& $\frt\in \Spin^c(Y,L)_\frs$ & $\Z^{|\ve{w}|}\oplus \Z^{|L|}$-filtered $\Z_2[\bar{U}_{\ve{w}}]$-module&
     $U_{\ve{w}}^I V_{\ve{z}}^J\cdot\ve{x}$ sat. Eq. \eqref{eq:relativeSpin^ccondition}\\ \hline
     $CFK^\infty(K)$ & $(K,w,z)\subset S^3$ & $\Z\oplus \Z$-filtered, $\Z_2[\bar{U}_w]$-module&$U_w^i V_z^j \cdot \ve{x}$ sat. $A(\ve{x})=j-i$\\
     \hline     $CFL^\infty(Y,\bL,\frs)$& $\frs\in \Spin^c(Y)$& $\Spin^c(Y,L)_\frs$-filtered, $\Z_2[U_{\ve{w}}]$-module&$U_{\ve{w}}^I \cdot \ve{x}$
     \\ \hline
    
    \end{tabular}
\end{center}

We will mostly use $CFL_{UV}^\infty(Y,\bL,\frs)$ and $CFL^\infty(Y,\bL,\frt)$, and will describe them more carefully in the following subsections. The object $CFL^\infty(Y,\bL,\frs)$ for $\frs\in \Spin^c(Y)$, introduced in \cite{OSMulti}, will not be useful for our purposes, because $\Psi_z$ and $\Phi_w$ are not filtered maps on this complex. Indeed we can think of $\Psi_z$ and $\Phi_w$ as being maps between $CFL^\infty(Y,\bL,\frt)$ for different choices of $\frt$. The object $CFL_{UV}^\infty(Y,\bL,\frs)$ is essentially just the $\Z_2$ direct sum over all $CFL^\infty(Y,\bL,\frt)$ ranging over relative $\Spin^c$ structures $\frt$ which restrict to $\frs$, so $CFL_{UV}^\infty$ becomes the most convenient object to work with.

\subsection{The complex \texorpdfstring{$CFL_{UV}^\infty(Y,\bL,\frs)$}{CFL\_UV(Y,L,s)}} Here we describe the uncolored complex $CFL_{UV}^\infty(Y,\bL,\frs)$. We first describe an intermediate object, $CFL_{UV,0}^\infty(Y,\bL,\frs)$.

Let $\Z_2[U_{\ve{w}},U_{\ve{w}}^{-1},V_{\ve{z}},V_{\ve{z}}^{-1}]$ denote the abelian algebra generated by variables $U_{w},U_w^{-1},V_{z}$ and $V_z^{-1}$ for $w\in \ve{w}$ and $z\in \ve{z}$. Given a diagram $\cH=(\Sigma, \ve{\alpha},\ve{\beta},\ve{w},\ve{z})$ for $(Y,L,\ve{w},\ve{z})$, we define $CFL_{UV,0}^\infty(\cH,\frs)$ to be the free \sloppy $\Z_2[U_{\ve{w}},U_{\ve{w}}^{-1},V_{\ve{z}},V_{\ve{z}}^{-1}]$-module  generated by $\ve{x}\in \bT_\alpha\cap \bT_\beta$ with $\frs_{\ve{w}}(\ve{x})=\frs$. We refer the reader to, e.g., \cite{OSMulti}, for the definition of a Heegaard diagram for a link, though we emphasize that in light of the results of \cite{JTNaturality}, we must assume that
 \[\ve{w}\cup \ve{z}\subset \Sigma\subset Y\] and that the embedding of $\Sigma$ in $Y$ is part of the data of a Heegaard splitting.

  We now define a map
 \[\d:CFL_{UV,0}^\infty(\cH,\frs)\to CFL_{UV,0}^\infty(\cH,\frs)\] by 
\[\d(\ve{x})=\sum_{\ve{y}\in \bT_\alpha\cap \bT_\beta} \sum_{\substack{\phi\in \pi_2(\ve{x},\ve{y})\\ \mu(\phi)=1}} \# \hat{\cM}(\phi)U_{\ve{w}}^{n_{\ve{w}}(\phi)}V_{\ve{z}}^{n_{\ve{z}}(\phi)}\cdot \ve{y}.\]

The map $\d$ does not square to zero,  but we do have the following:
\begin{lem}\label{lem:del^2=}The map $\d:CFL_{UV,0}^\infty(\cH,\frs)\to CFL_{UV,0}^\infty(\cH,\frs)$ satisfies
\[\d^2=\sum_{K\in |L|} (U_{w_{K,1}}V_{z_{K,1}}+V_{z_{K,1}}U_{w_{K,2}}+U_{w_{K,2}}V_{z_{K,2}}+\cdots V_{z_{K,n_K}}U_{w_{K,1}}),\]
 where $w_{K,1},z_{K,1},\dots, w_{K,n_K}, z_{K,n_K}$ are the basepoints on the link component $K$, in the order that they appear on $K$.\end{lem}
\begin{proof}This follows from the usual proof that the differential squares to zero, now just counting boundary degenerations carefully. If there are exactly two basepoints, there are no boundary degenerations by \cite{OSMulti}*{Thm. 5.5}, and the above formula is satisfied. If there are more than two, then each $\ve{\alpha}$- and $\ve{\beta}$-degeneration has a unique holomorphic representative by \cite{OSMulti}*{Thm. 5.5} and each crosses over a $\ve{w}$-basepoint and a $\ve{z}$-basepoint. The formula follows.
\end{proof}

To get a chain complex, we must color $CFL_{UV,0}^\infty(Y,\bL,\frs)$ by setting certain variables equal. Let $C_\bL$ denote the ideal generated by elements of the form $V_{z_i}-V_{z_j}$ where $z_i$ and $z_j$ are in the same link component. We let $\cL=\Z_2[U_{\ve{w}},U_{\ve{w}}^{-1}, V_{\ve{z}},V_{\ve{z}}^{-1}]/C_{\bL}$.

We now define 
\[CFL_{UV}^\infty(\cH,\frs)=CFL_{UV,0}^\infty(\cH,\frs)\otimes_{\Z_2[U_{\ve{w}}, U_{\ve{w}}^{-1},V_{\ve{z}},V_{\ve{z}}^{-1}]} \cL.\] We have the following:

\begin{lem}The map $\d$ defined above is a differential on $CFL_{UV}^\infty(\cH,\frs)$, i.e. $\d^2=0$.
\end{lem}
\begin{proof}This follows from the formula in Lemma \ref{lem:del^2=} since the module $\cL$ simply identifies all $V_z$ variables for $z$ which lie in the same link component.
\end{proof}

\begin{rem}There are other modules that we could tensor with to make the differential square to zero. The module $\cL$ is actually a quite natural choice. As we will see in the proof of Proposition \ref{lem:differentialcomp}, terms of the form $V_{z}+V_{z'}$ appear in the differential after quasi-stabilization, and these terms must be zero for $S_{w,z}^{\pm}$ to be chain maps.
\end{rem}

The $\Z_2[U_{\ve{w}},V_{\ve{z}}]$-module $CFL_{UV}^\infty(\cH,\frs)$ has a natural $\Z^{|\ve{w}|}\oplus \Z^{|L|}$ filtration given by filtering over powers of the variables $U_{\ve{w}}$ and $V_{\ve{z}}$.

There are of course many different Heegaard diagrams $\cH$ for a given multi-based link $(L,\ve{w},\ve{z})$. As in the case of closed three manifolds, using \cite{JTNaturality}, given two diagrams $\cH$ and $\cH'$, there is a $\Z^{|\ve{w}|}\oplus \Z^{|L|}$-filtered map
\[\Phi_{\cH\to \cH'}:CFL_{UV}^\infty(\cH,\frs)\to CFL^\infty_{UV}(\cH',\frs)\] which is a filtered chain homotopy equivalence, and is an invariant up to $\Z^{|\ve{w}|}\oplus \Z^{|L|}$-filtered chain homotopy. The maps $\Phi_{\cH\to \cH'}$ are functorial in the sense that if $\cH,\cH'$ and $\cH''$ are three diagrams, then
\[\Phi_{\cH'\to \cH''}\circ \Phi_{\cH\to \cH'}\simeq \Phi_{\cH\to \cH''}.\] The strongest invariant, which we will occasionally refer to as the \textbf{coherent filtered chain homotopy type}, is the collection of all of the complexes $CFL_{UV}^\infty(\cH,\frs)$ for all admissible diagrams, $\cH$, for $(Y,L,\ve{w},\ve{z})$, as well as the maps $\Phi_{\cH\to \cH'}$. We let
\[CFL_{UV}^\infty(Y,\bL,\frs)\] denote this invariant. Note that since we are working with embedded Heegaard surfaces, the set of Heegaard diagrams for a link is a set, and not a proper class.

\subsection{The complex \texorpdfstring{$CFL^\infty(Y,\bL,\frt)$}{CFL\^{}(Y,L,t)}}

We describe a complex $CFL^\infty(Y,\bL,\frt)$ where $\frt$ is a relative $\Spin^c$ structure whose underlying (nonrelative) $\Spin^c$ structure on $Y$ is $\frs$. This is defined as the subcomplex of $CFL^\infty_{UV}(\cH,\frs)$ generated over $\Z_2$ by elements of the form $\ve{x}\cdot U_{\ve{w}}^I V_{\ve{z}}^J$, where
\begin{equation}J\cdot PD(M)=(\frt-\frs_{\ve{w},\ve{z}}(\ve{x}))+I\cdot PD[M].\label{eq:relativeSpin^ccondition}\end{equation} Here if $J=(j_1,\dots, j_\ell)$ then $J\cdot PD(M)$ denotes $j_1\cdot PD[\mu_1]+\cdots j_\ell\cdot PD[\mu_\ell]$ and $I\cdot PD[M]$ is defined similarly. 

In the case that $\bL$ is a knot with exactly two basepoints, we see that $CFL^\infty(Y,L,w,z,\frt)$ is generated by elements of the form $\ve{x}\cdot U_w^i V_z^j$ with
\[j\cdot PD[\mu]=(\frt-\frs_{\ve{w},\ve{z}}(\ve{x}))+i\cdot PD[\mu],\] which is exactly the complex $CFK^\infty(\cH,\frt)$ found in \cite{OSKnots}. More often one writes $[\ve{x},i,j]$ for what we write $\ve{x}\cdot U_w^{-i} V_z^{-j}$. Most authors also write $U$ for the action defined by $U\cdot [\ve{x},i,j]=[\ve{x},i-1,j-1]$, though in our notation this action corresponds to multiplication by $\bar{U}_w=U_w V_{z}$.

 Given a relative $\Spin^c$ structure $\frt\in \Spin^c(Y,L)$, we write
 \[\iota_{\frt}:CFL^\infty(Y,L,\frt)\hookrightarrow CFL_{UV}^\infty(Y,L,\frs)\] for inclusion. 

 As a direct sum of $\Z_2$-modules, we have 
\[CFL_{UV}^\infty(\cH,\frs)= \bigoplus_{\frt\in \Spin^c(Y,L)_\frs} CFL^\infty(\cH,\frt).\] Write $\pi_{\frt}:CFL_{UV}^\infty(\cH,\frs)\to CFL^\infty(\cH,\frt)$ for the projection onto $CFL^\infty(\cH,\frt)$.

The object $CFL^\infty(Y,\bL,\frt)$ defined above is also isomorphic as a filtered chain complex to the object $CFL^\infty(Y,\bL,\frs)$ defined in \cite{OSMulti}. Given any $\frt\in \Spin^c(Y,L)$ there is a map 
 
 \[\phi_{\frt}:CFL^\infty(Y,\bL,\frs)\to CFL^\infty(Y,\bL,\frt)\] defined by
 \[\phi_{\frt}(U_{\ve{w}}^I\cdot \ve{x})=   U_{\ve{w}}^I V_{\ve{z}}^J\cdot\ve{x}\] where $J$ is the multi-index satisfying Equation \eqref{eq:relativeSpin^ccondition}.  The map $\phi_{\frt}$ is a $\Z_2[\bar{U}_{\ve{w}}]$-module isomorphism, where $\bar{U}_{w}$ acts as $U_w$ on $CFL^\infty(Y, \bL,\frs)$ and as $U_wV_{z}$ on $CFL^\infty(Y,\bL,\frt)$ (for any $z$ on the same link component as $w$). Given a choice of $\frt\in \Spin^c(Y,L)_\frs$, the $\Z^{|L|}$-affine space $\Spin^c(Y,L)_\frs$ can be identified with $\Z^{|L|}$. Such an identification gives a $\Z^{|L|}$ filtration on $CFL^\infty(Y,\bL,\frs)$, the object described in \cite{OSMulti}. This filtration is exactly the pullback under $\iota_{\frt}$ of the filtration on $CFL_{UV}^\infty$ by the $V_{\ve{z}}$ variables.

\subsection{The category of \texorpdfstring{$\frP$}{P}-filtered chain complexes}

We now define the category of $\frP$-filtered chain complexes.

\begin{define}A $\frP$-filtered chain complex is a chain complex with a filtration of $\Z^{\frP}$, i.e. if $C$ is a chain complex, then a $\frP$-filtration is a collection of submodules $\cF_I\subset C$ ranging over $I\in \Z^{\frP}$ such that if $I\le I'$, then $\cF_{I'}\subset \cF_{I}$. A $\frP$-filtered homomorphism is a homomorphism $\phi:C\to C'$ where $C$ and $C'$ are $\frP$-filtered with filtrations $\cF_I$ and $\cF'_{I}$ such that 
\[\phi(\cF_I)\subset \cF'_{I}.\]
\end{define}

Working with $\frP$-filtered chain complexes is easier than working with $\Spin^c(Y,L)$-filtered chain complexes for two reasons. Firstly it is hard to compute relative $\Spin^c$ structures. Secondly, when defining maps between filtered complexes, it is important that they have compatible filtrations, so that the notion of a filtered map can be defined. If we work with a filtration over just the $U_{\ve{w}}$ or $V_{\ve{z}}$ variables, it is challenging to describe a filtered map between complexes associated to different links or different collections of basepoints. Working with a fixed set $\frP$ allows us to do this.

\subsection{Colorings and \texorpdfstring{$\frP$}{P}-filtrations} As was the case in \cite{HFplusTQFT}, to define functorial maps, it is important to work in a fixed category. A coloring is a pair $(\sigma, \frP)$ where $\sigma:\ve{w}\cup \ve{z}\to \frP$ is a map which sends all of the $\ve{z}$ basepoints on a given link component to the same color (this condition ensures that the differential squares to zero).

Let $\cC_{\sigma,\frP}$ denote the module $\Z_2[U_{\ve{w}},U_{\ve{w}}^{-1},V_{\ve{z}},V_{\ve{z}}^{-1},U_{\frP},U_{\frP}^{-1}]/I_{\sigma, \frP}$ where $I_{\sigma,\frP}$ is the submodule generated by elements of the form $U_w-U_{\sigma(w)}$ and $V_z-U_{\sigma(z)}$. The colored complex is then defined as
\[CFL_{UV}^\infty(\cH, \sigma, \frP,\frs)=CFL_{UV,0}^\infty(\cH,\sigma)\otimes_{\Z_2[U_{\ve{w}},U_{\ve{w}}^{-1},V_{\ve{z}},V_{\ve{z}}^{-1}]} \cC_{\sigma, \frP}.\]

\begin{define}Given a coloring $(\sigma, \frP)$ of $\ve{w}\cup \ve{z}$ for the link $\bL=(L,\ve{w},\ve{z})$ in $Y$, the $\frP$-\textbf{filtration} to be the filtration on $CFL^\infty_{UV}(Y,\bL,\sigma, \frP,\frs)$ induced by powers of the variables $U_{\frP}$. An element of $CFL^\infty_{UV}(Y,\bL,\sigma, \frP,\frs)$ is uniquely written as a sum of elements of the form $\ve{x}\cdot U_{\frP}^I$, and given an $I\in \Z^\frP$ we define $\cF_I$ to be the $\Z_2[U_\frP]$-submodule generated by $\ve{x}\cdot U_{\frP}^J$ with $J\ge I$
\end{define}

\begin{rem}Asking that a map $F:CFL^\infty_{UV}(\cH,\sigma, \frP,\frs)\to CFL^\infty_{UV}(\cH',\sigma',\frP,\frs')$ be $\frP$-filtered is just asking that $F$ can be written as
\[F(\ve{x})=\sum_{I\ge 0}U_{\ve{\frP}}^I\cdot H_I(\ve{x}) \] where the maps $H_I$ do not involve the $U_{\ve{\frP}}$ variables. Most maps which appear in Heegaard Floer homology are $\frP$-filtered. The differential, the triangle maps, the quadrilateral maps, as well as the maps $\Phi_w,$ $\Phi_z$ and $S_{w,z}^{\pm}$ are all $\frP$-filtered.
\end{rem}

Given an arbitrary coloring $(\sigma, \frP)$ of basepoints $\ve{w}\cup \ve{z}$, we may not always be able to define submodules corresponding to relative $\Spin^c$ structures $\frt$. However if no two basepoints from distinct link components are given the same color, then one can use a modification of Equation \eqref{eq:relativeSpin^ccondition} to define a $\frP$-filtered $\Z_2$-submodule $CFL^\infty(Y,\bL,\sigma, \frP,\frt)$. For our purposes, we just observe that in the case that $(\sigma, \frP)$ is the trivial coloring (i.e. $\frP=\ve{w}\cup |L|$ and $\sigma$ is the map sending $w\in \ve{w}$ to $w$ and $z\in \ve{z}$ to the link component containing it), then $CFL_{UV}^\infty(\cH,\sigma,\frP,\frs)$ is equal to just $CFL_{UV}^\infty(\cH,\frs)$ and the maps $\pi_{\frt}$ and $\iota_{\frt}$ are still defined. The following lemma is essentially trivial, though it is useful for relating endomorphisms of $CFL_{UV}^\infty(Y,\bL,\frs)$ to endomorphisms of the subcomplexes $CFL^\infty(Y,\bL,\frt)$:

 \begin{lem}\label{lem:chainhomotopyonsubcomplex}Suppose that $(L,\ve{w},\ve{z})$ is a link in $Y$ and that $(\sigma,\frP)$ is the trivial coloring. Suppose $f$ and $g$ are $\frP$-filtered $\Z_2[U_{\frP}]$-module endomorphisms of $CFL_{UV}^\infty(Y,\bL,\sigma, \frP,\frs)$, such that $f$ and $g$ are chain homotopic via a chain homotopy which is $\frP$-filtered on $CFL_{UV}^\infty(Y,\bL,\sigma, \frP,\frs)$. Then $\pi_{\frt}\circ f\circ \iota_{\frt}$ and $\pi_{\frt}\circ g\circ \iota_{\frt}$ are $\Z^{|\ve{w}|}\oplus \Z^{|L|}$-filtered $\Z_2[\bar{U}_{\ve{w}}]$-chain homotopic.
 \end{lem}
 
 \begin{proof} First note that the filtration on $CFL^\infty(\cH,\frt)$ is just the pullback under $\iota_{\frt}$ of the $\frP$-filtration on $CFL_{UV}^\infty(\cH,\sigma,\frP,\frs)$. If $f$ and $g$ are $\frP$-filtered chain homotopic, we have that
 \[f-g=\d H+H\d\] for a $\frP$ filtered map $H$. Pre- and post-composing with the $\frP$-filtered maps $\iota_{\frt}$ and $\pi_{\frt}$ yields a $\Z^{|\ve{w}|}\oplus \Z^{|L|}$-filtered chain homotopy between  $\pi_{\frt}\circ f\circ \iota_{\frt}$ and $\pi_{\frt}\circ g\circ \iota_{\frt}$ because the maps $\pi_{\frt}$ and $\iota_{\frt}$ are $\frP$-filtered chain maps. The chain homotopy is $\Z_2[\bar{U}_{\ve{w}}]$-equivariant since the maps $\iota_{\frt}$ and $\pi_{\frt}$ are $\Z_2[\bar{U}_{\ve{w}}]$-equivariant, as we mentioned above.
 \end{proof}

\subsection{Why we use the larger \texorpdfstring{$CFL_{UV}^\infty(Y,L,\frs)$}{CFL\_UV(Y,L,s)} instead of \texorpdfstring{$CFL^\infty(Y,L,\frs)$}{CFL\^{}infty(Y,L,s)}}

We briefly explain why we use the object $CFL_{UV}^\infty(Y,\bL,\frs)$ to prove formulas for basepoint moving maps, instead of just the $\Spin^c(Y,L)_\frs$-filtered object $CFL^\infty(Y,\bL,\frs)$ from \cite{OSMulti}. This is simply because the maps $\Phi_w$ and $\Psi_z$ are not $\Spin^c(Y,L)_\frs$-filtered endomorphisms of $CFL^\infty(Y,L,\frs)$. Instead it's more convenient to think of them as maps between different $CFL^\infty(Y,L,\frt_i)$, and hence the most convenient operation is then to just take the direct sum over all $\frt_i$ by adding in the formal $V_{\ve{z}}$ variables. The map $\Phi_w$ on $CFL^\infty(Y,L,\frs)$ (the version of link Floer homology defined in \cite{OSMulti}, which we won't use) could be defined as \[\Phi_w(\ve{x})=U_w^{-1} \sum_{\ve{y}\in \bT_\alpha\cap \bT_\beta} \sum_{\substack{\phi\in \pi_2(\ve{x},\ve{y})\\ \mu(\phi)=1}} n_w(\phi)\# \hat{\cM}(\phi)U_{\ve{w}}^{n_{\ve{w}}(\phi)}\cdot \ve{y},\] though this may not be a filtered map because $U_w^{-1}$ is not a filtered map. In fact, it is easy to find examples where it is not filtered. By including the $V_{\ve{z}}$ variables, in the next section we consider the map
\[\Phi_w(\ve{x})=U_w^{-1} \sum_{\ve{y}\in \bT_\alpha\cap \bT_\beta} \sum_{\substack{\phi\in \pi_2(\ve{x},\ve{y})\\ \mu(\phi)=1}} n_w(\phi)\# \hat{\cM}(\phi)U_{\ve{w}}^{n_{\ve{w}}(\phi)}V_{\ve{z}}^{n_{\ve{z}}(\phi)}\cdot \ve{y},\] on $CFL_{UV}^\infty(Y,\bL,\frs)$, which is $\frP$-filtered, and is invariant up to $\frP$-filtered $\Z_2[U_\frP]$-chain homotopy (Lemma \ref{lem:PhiPsiinvaraints} below).

The trouble with using a map which is not filtered in the category of filtered chain complexes can be described as follows. The composition of filtered maps which are invariant up to filtered chain homotopy is itself invariant up to filtered chain homotopy. We could conceivably consider a map which is not filtered but is invariant up to filtered chain homotopy, though a composition involving an unfiltered map cannot be defined up to filtered chain homotopy, even if all of the maps involved are defined up to filtered chain homotopy. By working in $CFL_{UV}^\infty(Y,\bL,\frs)$, all of the maps we'll work with are $\frP$-filtered and defined up to filtered chain homotopy.

Additionally the map $\Psi_z$ has a convenient expression involving differentiating with respect to $V_z$, so it is somewhat easier to work over $CFL_{UV}^\infty(Y,\bL,\frs)$ than over the various $CFL^\infty(Y,\bL,\frt_i)$.

\section{The maps \texorpdfstring{$\Phi_w$}{Phi\_w} and \texorpdfstring{$\Psi_z$}{Psi\_z}}
\label{sec:PhiPsi}
Here we now define maps $\Phi_w$ and $\Psi_z$, which are endomorphisms of $CFL_{UV}(Y,\bL,\sigma,\frP,\frs)$. These are denoted by $\Phi_{i,j}$ and $\Psi_{i,j}$ in \cite{Smovingbasepoints}. We define $\Phi_w:CFL_{UV}^\infty(\cH,\sigma,\frP,\frs)\to CFL_{UV}^\infty(\cH,\sigma,\frP,\frs)$ by the formula
\[\Phi_w(\ve{x})=U_w^{-1} \sum_{\ve{y}\in \bT_\alpha\cap \bT_\beta} \sum_{\substack{\phi\in \pi_2(\ve{x},\ve{y})\\ \mu(\phi)=1}} n_w(\phi)\# \hat{\cM}(\phi)U_{\ve{w}}^{n_{\ve{w}}(\phi)}V_{\ve{z}}^{n_{\ve{z}}(\phi)}\cdot \ve{y},\] which we can alternatively think of as $(\frac{d}{d U_w} \d)$. Similarly we define
\[\Psi_z(\ve{x})=V_z^{-1} \sum_{\ve{y}\in \bT_\alpha\cap \bT_\beta} \sum_{\substack{\phi\in \pi_2(\ve{x},\ve{y})\\ \mu(\phi)=1}} n_z(\phi)\# \hat{\cM}(\phi)U_{\ve{w}}^{n_{\ve{w}}(\phi)}V_{\ve{z}}^{n_{\ve{z}}(\phi)}\cdot \ve{y},\] which we can alternatively write as  $(\frac{d}{d V_z} \d)$, where the derivative is taken on $CFL_{UV,0}^\infty$ (i.e. before tensoring with the module $\cL$ and thus setting all of the $V_z$ on a link component equal to each other).

We have the following (compare \cite{Smovingbasepoints}*{Lem. 4.1}):
\begin{lem}\label{lem:PhiPsicommutewithchangeofdiagrams}On $CFL_{UV}^\infty(\cH,\frs)$, we have $\Phi_w\d+\d \Phi_w=0$. Also $\Psi_z\d+\d \Psi_z=U_{w}+U_{w'}$, where $w$ and $w'$ are the $\ve{w}$ basepoints adjacent to $z$.
\end{lem}

\begin{proof} One takes the $V_z$ or $U_w$ derivative of $\d\circ\d$, before one tensors $CFL_{UV,0}^\infty$ with $\cL$. The map $\d^2$ is computed in Lemma \ref{lem:del^2=}. Then one tensors with $\cL$ to get the equality described above.
\end{proof}

In addition, we have the following (compare \cite{Smovingbasepoints}*{Theorem 4.2.}):
\begin{lem}\label{lem:PhiPsiinvaraints}The maps $\Phi_w$ and $\Psi_z$ commute with change of diagram maps $\Phi_{\cH_1\to \cH_2}$ up to $\frP$-filtered, $\Z_2[U_\frP]$-chain homotopy.
\end{lem}

\begin{proof}Consider the complex $CFL_{UV,0}^\infty$ (i.e. the complex before we set all of the $V_{z}$ variables on each link component equal to each other). The differential doesn't square to zero, though we can still consider the maps $\Phi_{\cH_1\to \cH_2}$. These can be written as a composition of change of almost complex structure maps, triangle maps (corresponding to $\ve{\alpha}$- or $\ve{\beta}$- isotopies or handleslides), (1,2)-stabilization maps, as well as maps corresponding to isotoping the Heegaard surface inside of $Y$ via an isotopy which fixes $\bL$. We claim that the maps $\Phi_{\cH_1\to \cH_2}$  satisfy
\[\Phi_{\cH_1\to \cH_2}\d+\d \Phi_{\cH_1\to \cH_2}=0,\] even before tensoring with $\cL$. The maps $\Phi_{\cH_1\to \cH_2}$  are defined as a composition of maps which count holomorphic triangles (handleslides or isotopy maps), holomorphic disks with dynamic almost complex structure (change of almost complex structure maps) or maps which are defined via tautologies (stabilization and diffeomorphism). The maps which are defined by tautologies obviously satisfy $\d \phi+\phi\d=0$ before tensoring with $\cL$. The maps induced by counting disks with dynamic almost complex structure also satisfy $\d\phi+\phi\d=0$ before tensoring with $\d$, since that follows from a Gromov compactness argument. Maps induced by handleslides or isotopies of the $\ve{\alpha}$-curves take the form
\[\ve{x}\mapsto F_{\ve{\alpha}'\ve{\alpha}\ve{\beta}}(\Theta\otimes \ve{x})\] where $\Theta$ is the top degree generator of a complex $CFL_{UV}^-(\Sigma,\ve{\alpha}',\ve{\alpha},\ve{w},\ve{z})$ and $F_{\ve{\alpha}'\ve{\alpha}\ve{\beta}}$ is the map which counts holomorphic triangles. For this to be a chain map before tensoring with $\cL$, it is sufficient for $\d \Theta=0$ before tensoring with $\cL$, since the triangle map 
\[F_{\ve{\alpha}'\ve{\alpha}\ve{\beta}}:CFL_{UV,0}^-(\Sigma,\ve{\alpha}',\ve{\alpha},\ve{w},\ve{z})\otimes CFL_{UV,0}^{\infty}(\Sigma, \ve{\alpha},\ve{\beta},\ve{w},\ve{z})\to CFL_{UV,0}^\infty(\Sigma, \ve{\alpha}',\ve{\beta},\ve{w},\ve{z})\] is a chain map by a Gromov compactness argument. We note now that the diagram $(\Sigma,\ve{\alpha}',\ve{\alpha},\ve{w},\ve{z})$ represents an unlink embedded in $(S^1\times S^2)^{\#n}$ for some $n$, and this unlink has exactly two basepoints per link component. By the differential computation in Lemma \ref{lem:del^2=}, the complex
\[CFL_{UV,0}^-(\Sigma,\ve{\alpha}',\ve{\alpha},\ve{w},\ve{z})\] is a chain complex before tensoring with anything, and in particular the homology group $HFL_{UV,0}^-(\Sigma,\ve{\alpha}',\ve{\alpha},\ve{w},\ve{z})$, is well defined even before tensoring with anything. An easy computation shows that if $HFL_{UV,0,\max}^-$ denotes the subset of maximal homological grading, then we have an isomorphism
\[HFL_{UV,0,\max}^-(\Sigma,\ve{\alpha}',\ve{\alpha},\ve{w},\ve{z})\iso \Z_2[V_{\ve{z}}],\] and in particular $HFL_{UV,0}^-(\Sigma,\ve{\alpha}',\ve{\alpha},\ve{w},\ve{z})$ admits a ``generator'' $\Theta$ which is distinguished by the relative Maslov grading and the algebraic property of generating the maximally graded subset as a $\Z_2[V_{\ve{z}}]$ module. In particular $\d \Theta=0$ even before tensoring with $\cL$, as we needed.

 Hence 
\[\Phi_{\cH_1\to \cH_2}\d+\d \Phi_{\cH_1\to\cH_2}=0,\] even before tensoring with $\cL$. Differentiating with respect to $U_w$ yields that
\[\Phi_{\cH_1\to \cH_2}'\d+\Phi_{\cH_1\to \cH_2} \Phi_w+\Phi_w\Phi_{\cH_1\to \cH_2}+\d \Phi_{\cH_1\to \cH_2}'=0,\] immediately implying that $\Phi_{\cH_1\to \cH_2} \Phi_w+\Phi_w\Phi_{\cH_1\to \cH_2}\simeq 0$. The only point to check is that the chain homotopy $H=\Phi_{\cH_1\to \cH_2}'$ is $\frP$-filtered and $\Z_2[U_\frP]$-equivariant. The equivariance condition is trivial. The filtration condition is also easy to check, since whenever $F$  has a decomposition with only nonnegative powers of $U_w$ and $V_z$, the map $(\frac{d}{d U_w} F)$ also has a decomposition with nonnegative powers of $U_w$ and $V_z$.
\end{proof}

\begin{rem}
Using the Leibniz rule we have that 
\[\Phi_w=\frac{d}{d U_w} \circ \d+\d \circ \frac{d}{d U_w},\] as long as $U_w$ doesn't share the same color as any other basepoint. Similarly $\Psi_z\simeq 0$ if $z$ doesn't share the same color with any other basepoint, though in both cases the chain homotopy $H=\frac{d}{d U_w}$ or $H=\frac{d}{d V_z}$ is neither $\frP$-filtered nor $\Z_2[U_\frP]$-equivariant.
\end{rem}
\section{Cylindrical boundary degenerations}

Suppose that $S$ is a Riemann surface with $d$ punctures $\{p_1,\dots, p_d\}$ on its boundary. We consider holomorphic maps 
\[u:S\to \Sigma\times (-\infty,1]\times \R\] such that the following hold:
\begin{enumerate}
\item $u$ is smooth;
\item $u(\d S)\subset (\ve{\alpha}\times \{1\}\times \R);$
\item $\pi_{\bD}\circ u$ is non-constant on each component of $S$;
\item for each $i$, $u^{-1}(\alpha_i\times \{1\}\times \R)$ consists of exactly one component of $\d S\setminus \{p_1,\dots, p_d\}$;
\item the energy of $u$ is finite;
\item $u$ is an embedding;
\item if $z_i\in S$ is a sequence of points approaching a puncture $p_j$, then $(\pi_\bD\circ u)(z_i)$ approaches $\infty$ in the compactification of $(-\infty,1]\times \R$ as the unit complex disk.
\end{enumerate}

We organize such curves into moduli spaces $\cN(\phi)$ for $\phi\in \pi_2^\alpha(\ve{x})$, modding out by automorphisms of the source, as usual. There is an action of $\PSL_2(\R)$ on $\cN(\phi)$, which is just the action on the $(-\infty,1]\times \R$ coordinate of a disk $u$, and we denote the quotient space $\hat{\cN}(\phi)$. One defines cylindrical $\ve{\beta}$-boundary degenerations analogously.
We now discuss transversailty. In the original set up (singly pointed diagrams and disks mapped into $\Sym^g(\Sigma)$), a generic almost complex structure $J$ on $\Sym^g(\Sigma)$ in a neighborhood of $\Sym^g(\frj)$ achieves transversality for Maslov index 2 holomorphic $\ve{\alpha}$-degenerate disks (\cite{OSDisks}*{Prop. 3.14}). In the cylindrical setup, if a sequence of holomorphic strips for the almost complex structures considered in \cite{LCylindrical} has a cylindrical boundary degeneration in its limit, the boundary degeneration will be $\frj_\Sigma\times \frj_\bD$-holomorphic. Thus we need transversality for cylindrical boundary degenerations for split almost complex structures. For the standard proof that $\d^2=0$ and for purposes of this paper, we only need transversality for Maslov index 2 boundary degenerations. Each of these domains has multiplicity 1 in one component of $\Sigma\setminus \ve{\alpha}$, and zero everywhere else. If $u:S\to \Sigma\times (-\infty,1]\times \R$ is a component of a holomorphic curve representing an element of $\cN(\phi)$ for a $\phi\in \pi_2^\alpha(\ve{x})$ such that $\pi_\Sigma\circ u$ is nonconstant, then by easy complex analysis $u|_{C}$ is injective, where $C\subset \d S$ is the part of $S$ mapping to $ \d \cD(\phi)$.   Adapting the strategy of perturbing boundary conditions instead of almost complex structures, as in \cite{LCylindrical}*{Prop. 3.9}, \cite{OSDisks}*{Prop. 3.9} or \cite{OhBoundary}, for generic choice of $\ve{\alpha}$-curves, we can thus achieve transversality for Maslov index 2 cylindrical boundary degenerations.

An important result for our purposes is a count of Maslov index 2 boundary degenerations produced by Ozsv\'{a}th and Szab\'{o}:

\begin{thm}[\cite{OSMulti}*{Thm. 5.5}]\label{thm:boundarydegencount}Consider a surface $\Sigma$ of genus $g$ equipped with a set of attaching circles $\ve{\alpha}=\{\alpha_1,\dots, \alpha_{g+\ell-1}\}$ which span a $g$-dimensional lattice in $H_1(\Sigma;\Z)$. If $\cD(\phi)\ge 0$ and $\mu(\phi)=2$, then $\cD(\phi)=A_i$ for some $i$, and indeed
\[\# \hat{\cN}(\phi)=\begin{cases}0 & (\Mod 2) \text{ if } \ell=1\\
1 & (\Mod 2) \text{ if } \ell>1
\end{cases}.\] Here $A_i$ denotes a component of $\Sigma\setminus \ve{\alpha}$.
\end{thm}

\section{Preliminaries on the quasi-stabilization operation}
\label{sec:prelimquasistab}
 Suppose that $\bL=(L,\ve{w},\ve{z})$ is an oriented link in $Y$ and that $w$ and $z$ are two points, both in a single component of $L\setminus (\ve{w}\cup \ve{z})$ such that $(L, \ve{w}\cup \{w\},\ve{z}\cup \{z\})$ has basepoints which alternate between $\ve{w}$ and $\ve{z}$ as one traverses the link. We assume that according to the orientation of $L$ the point $w$ comes after $z$. In Section \ref{sec:quasistabilizationnatural} we prove invariance for quasi-stabilization maps
\[S_{w,z}^+:CFL_{UV}^\infty(Y,L,\ve{w},\ve{z},\sigma,\frP,\frs)\to CFL_{UV}^\infty(Y,L,\ve{w}\cup \{w\},\ve{z}\cup \{z\},\sigma',\frP,\frs)\] and
\[S_{w,z}^-:CFL_{UV}^\infty(Y,L,\ve{w}\cup \{w\},\ve{z}\cup \{z\},\sigma',\frP,\frs)\to CFL_{UV}^\infty(Y,L,\ve{w},\ve{z},\sigma,\frP,\frs),\] which are defined up to  $\frP$-filtered $\Z_2[U_\frP]$-chain homotopy. Here $\sigma'$ is a coloring which extends $\sigma$. 

Though it will take several sections to construct the maps and prove they are well defined, we now summarize that the maps will be defined by the formulas
\[S_{w,z}^+(\ve{x})=\ve{x}\times \theta^+\] and
\[S_{w,z}^-(\ve{x}\times \theta^+)=0,\qquad \text{ and } \qquad S_{w,z}^-(\ve{x}\times \theta^-)=\ve{x},\] for suitable choices of Heegaard diagrams and almost complex structures.

To define the quasi-stabilization map, we use the special connected sum operation from \cite{MOIntSurg}. In \cite{MOIntSurg}, Manolescu and Ozsv\'{a}th describe a way of adding new $w$ and $z$ basepoints to Heegaard multi-diagrams. They prove that for multi-diagrams with at least three sets of attaching curves (e.g. Heegaard triples or quadruples), there is an identification of  certain moduli spaces of holomorphic curves on the unstabilized diagram and certain moduli spaces of holomorphic curves on the stabilized diagram. They conjecture an analogous result  for the holomorphic curves on a Heegaard diagram with two sets of attaching curves (i.e. for the differentials of quasi-stabilized diagrams), but only prove the result for grid diagrams using somewhat ad-hoc techniques. We will soon prove Proposition \ref{lem:differentialcomp} below, computing the differential on quasi-stabilized diagrams for appropriate almost complex structures.

\subsection{Topological preliminaries on quasi-stabilization}
 \label{subsec:topprelim}
 Suppose that $\cH=(\Sigma, \ve{\alpha},\ve{\beta},\ve{w},\ve{z})$ is a diagram for $(Y,L,\ve{w},\ve{z})$. Given new basepoints $w,z$ in the same component of $ L\setminus (\ve{w}\cup \ve{z})$, such that $w$ occurs after $z$, we now describe a new diagram $\bar{\cH}_{p,\alpha_s}$, which depends on a choice of point $p\in \Sigma$ and curve $\alpha_s\subset \Sigma\setminus \ve{\alpha}$. For fixed $\alpha_s$ and $p$, the diagram $\bar{\cH}_{p,\alpha_s}$ will be defined up to an isotopy of $Y$ relative $\ve{w}\cup \ve{z}\cup \{w,z\}$ which maps $L$ to $L$.
 
 Given a diagram $\cH=(\Sigma, \ve{\alpha},\ve{\beta},\ve{w},\ve{z})$ as above, let $A$ denote the component of $\Sigma\setminus \ve{\alpha}$ which contains the basepoints adjacent to $w$ and $z$ on $L$. Let $p\in A\setminus (\ve{\alpha}\cup \ve{\beta}\cup\ve{w}\cup \ve{z})$ be a point. If $U_\alpha$ denotes the handlebody component of $Y\setminus \Sigma$ such that the $\ve{\alpha}$-curves bound compressing disks in $U_\alpha$, then there is a path $\lambda$ in $U_\alpha$ from $p$ to a point on $L$ between $w$ and $z$. Such a curve $\lambda$ is specified up to an isotopy fixing $\ve{w}\cup \ve{z}\cup \{w,z\}$ which maps $L$ to $L$ by requiring that $\lambda$ be isotopic in $\bar{U}_\alpha$ to a segment of $L$ concatenated with an embedded arc on $\Sigma\setminus \ve{\alpha}$.

Let $N(\lambda)$ denote a regular neighborhood of $\lambda$ inside of $U_\alpha$ such that $\d N(\lambda)$, the boundary of $N(\lambda)$ inside of $U_\alpha$, satisfies
\[\d N(\lambda)\cap L=\{w,z\}.\] Topologically $\d N(\lambda)$ is just a disk, which we denote by $D_1$. Also let $D_2$ denote $N(\lambda)\cap \Sigma$, which we note is also a disk. We can assume that $D_2\cap (\ve{\alpha}\cup \ve{\beta}\cup \ve{w}\cup \ve{z})=\varnothing$. Define 
\[\bar{\Sigma}_{p}=(\Sigma\setminus D_2)\cup D_1.\] The surface $\bar{\Sigma}_{p}$ is specified up to an isotopy which fixes  $(\ve{\alpha}\cup \ve{\beta}\cup \ve{w}\cup \ve{z})$. Figure \ref{fig::15} shows the situation schematically. 
 
\begin{figure}[ht!]
\centering
\includegraphics[scale=1.2]{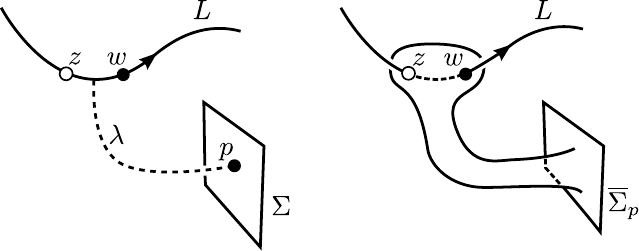}
\caption{The path $\lambda$ and the surfaces $\Sigma$ and $\bar{\Sigma}_p$. \label{fig::15}}
\end{figure}

We wish to extend the arc $\alpha_s\setminus D_2$ over all of $D_1$ to get a curve $\bar{\alpha}_s$ on $\bar{\Sigma}_p$. As is demonstrated in Figure \ref{fig::30}, there is not an isotopically unique way to do this relative the new basepoints $w$ and $z.$

\begin{figure}[ht!]
\centering
\includegraphics[scale=1.2]{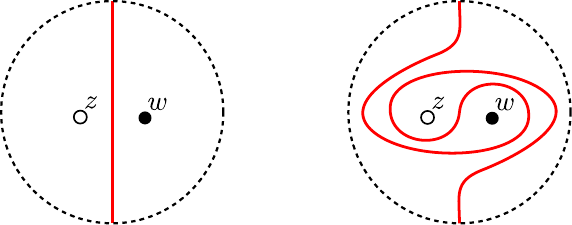}
\caption{Different choices of $\bar{\alpha}_s$ curve on $D_1$ interpolating $\alpha_s\setminus D_2$. There is a unique isotopy class of such curves such that the resulting $\bar{\alpha}_s$ curve on $\bar{\Sigma}_p$ bounds a compressing disk which doesn't intersect $L$.\label{fig::30}}
\end{figure}

The set of such curves is easily seen to be those generated by the images of the curve on the left in Figure \ref{fig::30} under finger moves of $w$ around $z$. Fortunately the arc $\alpha_s\setminus D_2$  can be extended over $D_1$ uniquely (up to isotopy) by requiring the resulting curve $\bar{\alpha}_s\subset \bar{\Sigma}_p$ to bound a compressing disk in $U_{\bar{\alpha}}$ which doesn't intersect $L$, where here $U_{\bar{\alpha}}$ denotes the component of $Y\setminus \bar{\Sigma}_p$ in which the $\ve{\alpha}$-curves bound compressing disks.

Given a diagram $\cH=(\Sigma, \ve{\alpha},\ve{\beta},\ve{w},\ve{z})$ and a point $p\in \Sigma$ as above, we can form a diagram 
\[\bar{\cH}_{p,\alpha_s}=(\bar{\Sigma}_p, \bar{\ve{\alpha}},\bar{\ve{\beta}},\ve{w}\cup \{w\}, \ve{z}\cup \{z\}),\] where $\bar{\ve{\alpha}}=\ve{\alpha}\cup \{\bar{\alpha}_s\}$ and $\bar{\ve{\beta}}=\ve{\beta}\cup \{\beta_0\}$. The curve $\beta_0$ is specified up to isotopy which fixes $(\Sigma\setminus D_2)\cup \{w,z\}$ though the exact choice of $\beta_0$ in principle depends on a choice of gluing data $\cJ$, described more precisely in the following section, but we suppress the dependence in the notation.  We write $\bar{\ve{\alpha}}$ for $\ve{\alpha}\cup \{\bar{\alpha}_s\}$, the curves on the stabilized diagram after performing special connected sum, and we write $\ve{\alpha}^+$ for $\ve{\alpha}\cup \{\alpha_s\}$, the curves on the unstabilized diagram on $\Sigma$. By abuse of notation, we will often write $\alpha_s$ to denote both $\alpha_s\subset \Sigma$ and  $\bar{\alpha}_s\subset \bar{\Sigma}_{p}$.

\begin{figure}[ht!]
\centering
\includegraphics[scale=1.2]{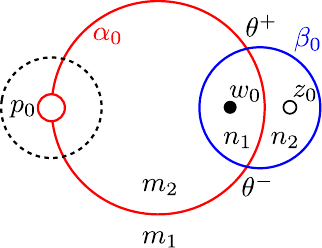}
\caption{The diagram $\cH_0$ used for quasi-stabilization, with multiplicities labelled. The dashed circle denotes where we will perform the neck stretching in the special connected sum. \label{fig::1}}
\end{figure}

\subsection{Gluing data and almost complex structures}\label{subsec:gluingdata}

In \cite{HFplusTQFT} the author describes a systematic way of proving invariance of free stabilization maps from the choices of almost complex structures. Here we introduce the analogous idea for quasi-stabilization.

Suppose that $\cH=(\Sigma, \ve{\alpha},\ve{\beta},\ve{w},\ve{z})$ is a diagram for $\bL=(L,\ve{w},\ve{z})$ and $w,$ and $z$ are two new consecutive basepoints on $L$ with $z$ following $w$. Suppose $p\in A\setminus (\ve{w}\cup \ve{z}\cup \ve{\alpha}\cup \ve{\beta})$ is a distinguished point, where here $A$ denotes the component of $\Sigma\setminus \ve{\alpha}$ containing the basepoints on $L$ adjacent to $w$ and $z$. Let $\bar{\Sigma}_p$ denote the Heegaard surface described in the previous subsection. We have the following definition:

\begin{define} We define \textbf{gluing data} to be a collection $\cJ=(\psi,J_s,J_{s,0},B,B_0,r,r_0,p_0,\iota,\phi)$ where
\begin{enumerate}
\item $\psi:\Sigma\to \bar{\Sigma}_p$ is a diffeomorphism which is fixed outside of $D_2$ and maps $\alpha_s$ to $\bar{\alpha}_s$;
\item $B\subset \Sigma$ is a closed ball containing $p$ which doesn't intersect $(\ve{w}\cup \ve{z}\cup \ve{\alpha}\cup \ve{\beta})$ and such that $\alpha_s\cap B$ is a closed arc;
\item the point $p_0\in \alpha_0\setminus \beta_0$ is the connected sum point;
\item $B_0\subset S^2$ is a closed ball containing $p_0$ which doesn't intersect $\beta_0$ and such that $B_0\cap \alpha_0$ is a closed arc;
\item $J_s$ is an almost complex structure on $\Sigma\times [0,1]\times \R$ which is split on $B$;
\item $J_{s,0}$ is an almost complex structure on $S^2\times [0,1]\times \R$ which is split on $B_0$;
\item $r$ and $r_0$ are real numbers such that $0<r,r_0<1;$
\item using the unique (up to rotation) conformal identifications of $(B,p)$ and $(B_0,p_0)$ as $(\bD,0)$, where $\bD$ denotes the unit complex disk, $\iota$ is an embedding of $\bar{S^2\setminus r_0\cdot B_0}$ into $r\cdot B\subset \Sigma$ such that
\[\iota(\alpha_0)\subset \alpha_s,\]
\[(\psi\circ\iota)(z_0)=z,\qquad (\psi\circ\iota)(w_0)=w,\] and
\[(r\cdot B)\setminus \iota(S^2\setminus r_0\cdot B_0)\] is a closed annulus;
\item  letting $\tilde{A},A$ and $A_0$ denote the closures of the annuli $B\setminus \iota (S^2\setminus B_0)$, $B\setminus r \cdot B$ and $B_0\setminus r_0 \cdot B_0$ respectively, then 
\[\phi:\tilde{A}\to S^1\times [-a,1+b]\] is a diffeomorphism which sends the annulus $A$ to $[-a,0]$ and $\iota(A_0)$ to $[1,1+b]$ and is conformal on $A$ and $A_0$.
\end{enumerate}
\end{define}

The space of embeddings $\iota$ is connected since if $a$ denotes the arc on the left side of Figure \ref{fig::30}, the space of diffeomorphisms $f:B\to B$ mapping $\d B\cup a$ to itself and fixing $\{z,w\}$ is connected. That the space of diffeomorphisms $\psi:\Sigma\to \bar{\Sigma}_p$ in the definition is also connected follows for similar reasons.

Gluing data $\cJ$ and a choice of neck length $T$ determines an almost complex structure $\cJ(T)$ on $\bar{\Sigma}_{p}\times [0,1]\times \R$.

\subsection{Computing the quasi-stabilized differential}
We now wish to compute the differential after performing quasi-stabilization. We have the following Maslov index computation:
\begin{lem}\label{lem:Maslovindexdisk}If $\phi_0$ is a homology class of disks on $\cH_0$, then using the multiplicities in Figure \ref{fig::1}, we have \[\mu(\phi_0)=(n_1+n_2+m_1+m_2)(\phi_0).\]
\end{lem}
\begin{proof} The formula is easily verified for the constant disk and respects splicing in any of the Maslov index 1 strips.
\end{proof}

Let $\cJ$ denote gluing data as in the previous section, and let $\cJ(T)$ denote the almost complex structure on $\bar{\cH}_{p,\alpha_s}$ determined by $\cJ$ for a choice of neck length $T$. We have the following:

\begin{prop}\label{lem:differentialcomp}Suppose that $\cH$ is a strongly $\frs$-admissible diagram and that $\cJ$ is gluing data with almost complex structure $J_s$ on $\Sigma\times [0,1]\times \R$. Then  $\bar{\cH}_{p,\alpha_s}$ is also strongly $\frs$-admissible and for sufficiently large $T$ there is an identification of uncolored differentials (i.e. before we tensor with $\cL$)
\[\d_{\bar{\cH}_{p,\alpha_s},\cJ(T)}=\begin{pmatrix}\d_{\cH,J_s}&U_{w}+U_{w'}\\
V_{z}+V_{z'}& \d_{\cH,J_s}
\end{pmatrix},\] where the basepoints $w$ and $z$ are placed between $w'$ and $z'$ on $\bL$. 
\end{prop}
\begin{proof}Suppose that $u_i$ is a sequence of Maslov index 1 curves on $\bar{\cH}_{p,\alpha_s}$ for the almost complex structure $\cJ(T_i)$ for a sequence of $T_i$ with $T_i\to \infty$. From the sequence $u_i$ we can extract a weak limit of curves on the diagrams $(\Sigma, \ve{\alpha}^+, \ve{\beta},\ve{w},\ve{z})$ and $\cH_0$. Let $U_\Sigma,$ and $U_0$ denote these collections of curves. The curves in $U_\Sigma$ consist of flowlines on $(\Sigma, \ve{\alpha}^+,\ve{\beta},\ve{w},\ve{z})$ as well as $\ve{\alpha}$- and $\ve{\beta}$-boundary degenerations on $(\Sigma, \ve{\alpha}^+)$ and $(\Sigma, \ve{\beta})$, and closed surfaces mapped into $\Sigma$. The holomorphic curves are now allowed to have a puncture along the $\alpha$-boundary which is mapped asymptotically to $p$.

We first note that any flowline in the limit (i.e. a map $u:S\to \Sigma\times [0,1]\times \R$ which maps $\d S$ to $(\ve{\beta}\times \{0\}\cup\ve{\alpha}^+\times \{1\}\times \R$ such that each component of $S$ has both $\ve{\alpha}^+$ and $\ve{\beta}$ components) on the diagram $\cH^+=(\Sigma, \ve{\alpha}^+,\ve{\beta},\ve{w},\ve{z})$ must actually be a legitimate flow line on $(\Sigma, \ve{\alpha},\ve{\beta},\ve{w},\ve{z})$. This is because if $u$ is any holomorphic curve which is part of a weak limit of the curves $u_{T_i}$, then $u$ cannot have a puncture asymptotic to an intersection point $\alpha_s\cap \beta_j$ for $\beta_j\in \ve{\beta}$. Hence if $u$ is part of the weak limit of $u_i$, and if $S$ denotes the source of $u$, then if $\d S$ has any points mapped to $\alpha_s$, then $ S$ must have boundary component with a single puncture which is mapped to $\alpha_s$. Projecting to $I\times \R$, we note that either $u|_{\d S}$ obtains a local extrema, or $u$ is asymptotic to both $+\infty$ and $-\infty$ as one approaches the puncture. The first case isn't possible since one can use the doubling trick to create an analytic function $\bD\to \bD$ (where here $\bD=\{z:|z|<1\}$) which maps $\bD\cap \{\im(z)\ge 0\}$ to $\bD\cap \{\im(z)\ge 0\}$ but which satisfies $f'(z)=0$ for some $z\in \R$, which is impossible by writing down a local model. The latter case is impossible since $u$ must extend to a continuous function over the punctures. Hence any such $u$ must be constant in the $I\times \R$ component, which implies that $u$ cannot have any portion of $\d S$ mapped onto a $\ve{\beta}$-curve. Hence the curves in the weak limit can be taken to be disks on $(\Sigma, \ve{\alpha},\ve{\beta},\ve{w},\ve{z})$, $\ve{\alpha}^+$- or $\ve{\beta}$-degenerations, or closed surfaces.

To avoid ``annoying'' curves (i.e. maps into $\Sigma\times [0,1]\times \R$ which are constant in the $I\times \R$-component) in the $\ve{\beta}$- or $\ve{\alpha}^+$-degenerations, we observe that by rescaling the $I\times \R$ component, we can instead get curves that map into $\Sigma\times (-\infty,1]\times \R$ or $\Sigma\times S^2$ such that the $(-\infty,1]\times\R$ or $S^2$ components are nonconstant. Maps into $\Sigma\times (-\infty,1]\times \R$ are cylindrical $\ve{\alpha}^+$-boundary degenerations. As remarked earlier, Maslov index two cylindrical boundary degenerations do achieve transversality for generic perturbation of the $\ve{\alpha}^+$-curves (in fact we need transversality of these curves to prove that the differential squares to zero).

We now wish to compute exactly which of the above degenerations can occur in a weak limit of the sequence $u_i$ of Maslov index one $\cJ(T_i)$-holomorphic curves. Assume without loss of generality that all of the $u_i$ are in the same homology class $\phi$. 

Suppose that $U_\Sigma$ consists of a collection $U_\Sigma'$ of curves on $(\Sigma,\ve{\alpha},\ve{\beta})$ (flowlines, boundary degenerations, closed surfaces) and a collection of curves $\cA$ in $(\Sigma, \ve{\alpha}^+,\ve{\beta})$ which have a boundary component which maps to $\alpha_s$. As we've already remarked, the collection $\cA$ consists exactly of cylindrical $\ve{\alpha}^+$-boundary degenerations. Letting $\phi_\Sigma'$ denote the underlying homology class of $U_\Sigma'$, we define a combinatorial Maslov index for $U_\Sigma$ by
\[\mu(U_\Sigma)=\mu(\phi_\Sigma')+m_1(\cA)+m_2(\cA)+2\sum_{\substack{\cD\in C(\Sigma\setminus \ve{\alpha}),\\ \alpha_s\cap \cD=\varnothing}} n_{\cD}(\cA),\] where $C(\Sigma\setminus \ve{\alpha})$ denotes the connected components of $\Sigma\setminus \ve{\alpha}$. By Lemma \ref{lem:Maslovindexdisk} the Maslov index of $U_0$ satisfies
\[\mu(U_0)=m_1(\phi)+m_2(\phi)+n_1(\phi)+n_2(\phi).\] The formula for $\mu(U_\Sigma)$ does not necessarily count the expected dimension of anything since we've only defined it combinatorially, though we can compute the Maslov index $\mu(\phi)$ using these formulas, as follows:
\[\mu(\phi)=\mu(U_\Sigma)+\mu(U_0)-m_1(\phi)-m_2(\phi)\]
\[=\mu(\phi_\Sigma')+n_1(\phi)+n_2(\phi)+m_1(\cA)+m_2(\cA)+2\sum_{\substack{\cD\in C(\Sigma\setminus \ve{\alpha})\\ \alpha_s\cap \cD=\varnothing}} n_{\cD}(\cA).\] In the above sum, each term is nonnegative, and by assumption the total sum is equal to one. If $\mu(\phi)=1$, we immediately have that $n_{\cD}(\cA)=0$ for $\cD\in C(\Sigma\setminus \ve{\alpha}),$ with $\alpha_s\cap \cD=\varnothing$. Now $\mu(\phi_\Sigma')$ does actually count the expected dimension of the moduli space of $\phi_\Sigma'$, and in particular if $\phi_\Sigma'$ has a representative as a broken curve, we must have $\mu(\phi_\Sigma')\ge 0$ with equality iff $\phi_\Sigma'$ is the constant disk. 

As a consequence we see that if $\mu(\phi)=1$, we have that exactly one of $\mu(\phi_\Sigma'),n_1(\phi),n_2(\phi),m_1(\cA),$ or $m_2(\cA)$ is 1, and the rest are zero. The cases of $n_1(\phi)=1$ or $n_2(\phi)=1$ respectively are easy to analyze, and those possibilities contribute summands of
\[\begin{pmatrix} 0& U_w\\
0& 0\end{pmatrix}\qquad \text{ and }\qquad  \begin{pmatrix}0&0\\
V_z& 0
\end{pmatrix},\] respectively, to $\d_{\bar{\cH}_{p,\alpha_s}}$.

We now consider broken curves in the limit with $\mu(\phi_\Sigma')=1$ and the remaining terms zero. In this case case we have that $m_1(\cA)=m_2(\cA)=0$ (so $\cA=0$) and $n_1(\phi)=n_2(\phi)=0$. In this case, we observe that $\phi_\Sigma'$ is represented by $U_\Sigma$ and $\mu(\phi_\Sigma')=1$, so the limit cannot contain any boundary degenerations or closed surfaces.  Furthermore, by Maslov index considerations we have that $U_\Sigma$ consists of a single Maslov index one flowline, which we denote by $u_\Sigma$.

 We now consider the curves in $U_0$. There must be a component of $U_0$ which satisfies a matching condition with $u_\Sigma$. Note also that since $U_\Sigma$ consists only of a single disk on $(\Sigma, \ve{\alpha},\ve{\beta})$, there cannot be any curves in $U_0$ which have a point on the boundary mapped to $p_0$ or be asymptotic to $p_0$.

Let $u_0$ denote the component of $U_0$ which satisfies the matching condition
\[\rho^{p}(u_\Sigma)=\rho^{p_0}(u_0).\] In particular, this forces $m_1(u_0)=m_2(u_0)$. We also have $n_1(u_0)=n_2(u_0)=0$. Here, if $u:S\to \Sigma\times [0,1]\times \R$ is a holomorphic disk,  $\rho^q(u)$ is the divisor \[(\pi_\bD\circ u)(\pi_\Sigma\circ u)^{-1}(q)\in \Sym^{n_q(u)}(\bD).\] Any additional components $u_0'$ of of $U_0$ must also satisfy $n_1(u_0')=n_2(u_0')=0$ and also can't have an interior point or boundary point mapped to $p_0$, and hence must be constant. Hence $U_0$ consists exactly of a holomorphic strip $u_0$ with $m_2(u_0)=m_1(u_0)$ which satisfies a matching condition with $u_\Sigma$, i.e. $(u_\Sigma,u_0)$ are prematched strips.

We thus have shown that if $\mu(\phi_\Sigma')=1$, then the weak limit of the curves $u_i$ is  a prematched strip, so following standard gluing arguments (cf. \cite{LCylindrical}*{App. A}), the count of $\hat{\cM}_{\cJ(T)}(\phi)$ is equal to the count of prematched strips with total homology class $\phi$, for sufficiently large $T$. If $\phi_0$ is a homology class of disks in $\pi_2(\theta^+,\theta^+)$ or $\pi_2(\theta^-,\theta^-)$, let $\cM(\phi_0,\ve{d})$ denote the set of holomorphic strips $u$ representing $\phi_0$ with $\rho^{p_0}(u)=\ve{d}$. Note that there is a unique homology class of disks $\phi_0\in \pi_2(\theta^+,\theta^+)$ with $m_1(\phi_0)=m_2(\phi_0)=|\ve{d}|$ and $n_1(\phi_0)=n_2(\phi_0)=0$.

 We claim that \[\cM(\phi_0,\ve{d})\equiv 1\pmod{2}\] if $m_1(\phi_0)=m_2(\phi_0)=|\ve{d}|$ and $n_1(\phi_0)=n_2(\phi_0)=0$. We consider a path $\ve{d}_t$ between two divisors $\ve{d}_0$ and $\ve{d}_1$ and consider the 1-dimensional space 
\[\cM=\coprod_{t\in I}\cM(\phi_0,\ve{d}_t).\] We count the ends of $\cM$. There are ends corresponding to $\cM(\phi_0,\ve{d}_0)$ and $\cM(\phi_0,\ve{d}_1)$. On the other hand, there are ends corresponding to strip breaking or other types of degenerations. Any curve in the degeneration cannot have $p_0$ in it's boundary, which constrains any degeneration to be into disks of the form $\pi_2(\theta^+,\theta^+)$ or $\pi_2(\theta^-,\theta^-)$. But if any nontrivial strip breaking occurs, the Maslov index of the matching component drops, contradicting the formula for the Maslov index. Hence the only ends of $\cM$ correspond to $\cM(\phi_0,\ve{d}_0)$ and $\cM(\phi_0,\ve{d}_1)$, implying that 
\[\# \cM(\phi_0,\ve{d}_0)\equiv \# \cM(\phi_0,\ve{d}_1) \pmod 2.\] We now consider a path of divisors $\ve{d}_T$ consisting of $k$ points in $I\times\R$ spaced at least $T$ apart which approach the line $\{0\}\times \R$ as $T\to \infty$. Letting $T\to \infty$, since $p_0$ is not on $\beta_0$, we know that the Gromov limit of the curves in $\cM(\phi_0,\ve{d}_T)$ consists of $k$ cylindrical $\beta_0$-degenerations, and a single constant holomorphic strip. Applying Theorem \ref{thm:boundarydegencount}, we get that the total count of the boundary of \[\overline{\sqcup_T\cM(\phi_0,\ve{d}_T)}\] is $\#\cM(\phi_0,\ve{d}_1)+1$, implying the claim.

Disks which consist of preglued flowline $(u_\Sigma, u_0)$ glued together thus provide a total contribution of
\[\begin{pmatrix}\d_{\cH,J_s}& 0\\
0& \d_{\cH, J_s}
\end{pmatrix},\] to the differential.

We now consider the last contribution to the differential. These correspond to having $\cD(\phi_1')=0$ and $n_1=n_2=0$, and $m_1(\cA)+m_2(\cA)=1$. In this case $\cA$ is an $\ve{\alpha}^+$-boundary degeneration on $(\Sigma, \ve{\alpha}^+,\ve{\beta})$. Since \[\sum_{\substack{\cD\in C(\Sigma\setminus \ve{\alpha})\\ \alpha_s\cap \cD=\varnothing}} n_{\cD}(\cA)=0,\] we have that $\cD(\cA)$ is constrained to one of two domains. For each of the two possible domains of $\cD(\cA)$, there is exactly one corresponding choice of domain for $\phi_0$, the homology class of $U_0$, which is just the domain with exactly one of $m_1$ and $m_2$ equal to one, and the other equal to zero, as well as $n_1$ and $n_2$ also zero.

To count such disks, note that each component of $U_\Sigma$ is constrained to be either a trivial disk or a cylindrical $\ve{\alpha}^+$-degeneration. The domain of $\phi_0$ constrains $U_0$ to be a single a Maslov index 1 flowline on $(S^2,\alpha_0,\beta_0)$ with a single puncture on the $\alpha$ boundary, of which there is a unique representative modulo the $\R$-action. Now the curve in $U_\Sigma$ which has nontrivial domain is a map from a punctured surface $u:S\to \Sigma\times (-\infty,1]\times \R$. At the punctures, $u$ is asymptotic to the points in $\{x_1,\dots,x_k,p\}$ where $x_1,\dots, x_k$ is a collection of intersection points of $\ve{\alpha}$- and $\ve{\beta}$-curves where exactly one $x_i$ is on each $\ve{\alpha}$-curve which intersects the domain $u$ nontrivially. By Theorem \ref{thm:boundarydegencount}, we know that the count of such curves is $1$ modulo 2. Since these homology classes are $\alpha$-injective, we know that by perturbing the $\ve{\alpha}^+$-curves, we achieve transversality. Hence by gluing we get that for sufficiently large neck length, the number of representatives of such curves is exactly 1. There are two domains which are counted by such curves, and with the above count we see that such curves make contributions of
\[\begin{pmatrix} 0& U_{w'}\\
0&0
\end{pmatrix}\qquad \text{ and } \qquad \begin{pmatrix}0& 0\\
V_{z'}& 0
\end{pmatrix},\] where $w$ and $z$ are immediately adjacent to $w'$ and $z'$ on $\bL$. Summing together all of the contributions, we see that the differential takes the form

\[\d_{\bar{\cH}_{p,\alpha_s},\cJ(T)}=\begin{pmatrix}\d_{\cH,J_s}& U_w+U_{w'}\\
V_z+V_{z'}& \d_{\cH, J_s}
\end{pmatrix}.\]
\end{proof}

Note that we tensor $CFL_{UV,0}^\infty(\cH,\frs)$ with $\cL$ so that the differential squares to zero. When we do this, we set $V_z=V_{z'}$, and the bottom left entry of the differential vanishes.

\subsection{Dependence of quasi-stabilization on gluing data}

In this subsection we prove some initial results about quasi-stabilization and change of almost complex structure maps. The reader should compare this to \cite{HFplusTQFT}*{Sec. 6}, where the analogous arguments are presented for free stabilization.

\begin{lem}\label{lem:transitionmapsstabilize}Suppose that $\cJ$ is gluing data. Then there is an $N$ such that if $T,T'>N$ and $\cJ(T),$ and $\cJ(T')$ achieve transversality, then
\[\Phi_{\cJ(T)\to \cJ(T')}\simeq \begin{pmatrix}1&*\\
0&1
\end{pmatrix}.\]
\end{lem}

The proof is analogous to the proof of \cite{HFplusTQFT}*{Lem. 6.8}, using the techniques of Proposition \ref{lem:differentialcomp}. Hence as in \cite{HFplusTQFT}, we make the following definition:

\begin{define}We say that $N$ is \textbf{sufficiently large} for gluing data $\cJ$, if $\Phi_{\cJ(T)\to \cJ(T')}$ is of the form in the previous lemma for all $T,T'\ge N$. We say $T$ is \textbf{large enough to compute} $S^\pm_{w,z}$ for the gluing data $\cJ$ if $T> N$ for some $N$ which is sufficiently large.
\end{define}

Adapting the proofs of \cite{HFplusTQFT}*{Lem. 6.10-6.16}, we have the following:

\begin{lem}\label{lem:changeofalmostcomplexstructure}If $\cJ$ and $\bar{\cJ}$ are two choices of gluing data with almost complex structures $J_s$ and $\bar{J}_s$ on $\Sigma\times [0,1]\times \R$, respectively, then there is an $N$ such that if $T>N$, and $\cJ(T)$ and $\bar{\cJ}(T)$ achieve transversality, then
\[\Phi_{\cJ(T)\to \bar{\cJ}(T)}\simeq\begin{pmatrix}\Phi_{J_s\to \bar{J}_s}&*\\
0& \Phi_{J_s\to \bar{J}_s}
\end{pmatrix}.\]
\end{lem}

The previous lemma will be used to show that the quasi-stabilization maps are independent of the choice of gluing data.

\section{Quasi-stabilization and triangle maps} In this section we prove several results about quasi-stabilizing Heegaard triples, which we will use to prove invariance of the quasi-stabilization maps. The results for quasi-stabilization of Heegaard triples along a single $\alpha_s$ curve are established in \cite{MOIntSurg}, so we focus on quasi-stabilizing a Heegaard triple along two curves, $\alpha_s$ and $\beta_s$. To compute the quasi-stabilization maps $S_{w,z}^{\pm}$, we pick a curve $\alpha_s$ in the surface $\Sigma$, but there are many choices of such an $\alpha_s$ curve, so in order to address invariance of the quasi-stabilization maps, we need to show that the maps $S_{w,z}^{\pm}$ commute with the change of diagram map corresponding to moving $\alpha_s$ to $\alpha_s'$, which can be computed using a Heegaard triple which has been quasi-stabilized along two curves.  Our main result is Theorem \ref{thm:doublequasitriangle}, which is an analog to our computation of the differential after quasi-stabilizing in Proposition \ref{lem:differentialcomp}, but for Heegaard triples which we have quasi-stabilized along two curves which are allowed to travel throughout the diagram.

\subsection{Quasi-stabilizing Heegaard triples along a single curve}

We now consider Heegaard triples which are quasi-stabilized along a single $\alpha_s$ curve which is allowed to run through the diagram. This was first considered in \cite{MOIntSurg}. We state \cite{MOIntSurg}*{Prop. 5.2}, which considers the quasi-stabilized configuration shown in Figure \ref{fig::16}. The result will be useful in showing that the quasi-stabilization maps are invariant under $\ve{\beta}$-handleslides and $\ve{\beta}$-isotopies.

\begin{figure}[ht!]
\centering
\includegraphics[scale=1.2]{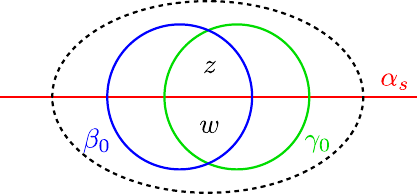}
\caption{The version of quasi-stabilization discussed in Lemma \ref{lem:singlealphaquasistabtriangle}, which is considered in \cite{MOIntSurg}*{Prop. 5.2}} \label{fig::16}
\end{figure}

\begin{lem}[\cite{MOIntSurg}*{Prop. 5.2}]\label{lem:singlealphaquasistabtriangle} Suppose that $\cT=(\Sigma,\ve{\alpha},\ve{\beta},\ve{\gamma},\ve{w},\ve{z})$ is a strongly $\frs$-admissible triple and suppose that $\alpha_s$ is a new $\ve{\alpha}$-curve, passing through the point $p\in \Sigma$. Let $\bar{\cT}_{\alpha_s,p}$ denote the Heegaard triple resulting from quasi-stabilizing along $\alpha_s$ at $p$, as in Figure \ref{fig::16}. If $\cJ$ is gluing data, then for sufficiently large $T$, with the almost complex structure $\cJ(T)$ we have the following identifications:

\[F_{\bar{\cT}_{\alpha_s,p},\bar{\frs}}(\ve{x}\times \cdot, \ve{y}\times y^+)=\begin{pmatrix}F_{\cT,\frs}(\ve{x},\ve{y})&0\\
0& F_{\cT,\frs}(\ve{x},\ve{y})
\end{pmatrix}\] where $\cdot\in \{x^+,x^-\}=\alpha_s\cap \beta_0$ and the matrix on the right denotes the expansion into the upper and lower generator components of $\alpha_s\cap \beta_0$ and $\alpha_s\cap\gamma_0$.
\end{lem}

\subsection{Strongly positive diagrams for \texorpdfstring{$(S^1\times S^2)^{\#k }$}{connected sums of S\^{} 1 times S\^{} 2}}

In this subsection we describe a class of simple diagrams for $(S^1\times S^2)^{\# k}$ which we will use in a technical condition in Theorem \ref{thm:doublequasitriangle} for quasi-stabilizing Heegaard triples along two curves, $\alpha_s$ and $\beta_s$, passing through a Heegaard triple.

Suppose that $(\Sigma, \ve{\alpha},\ve{\beta},\ve{w})$ is a diagram for $(S^1\times S^2)^{\# k}$ such that 
\[|\alpha_i\cap \beta_j|=\begin{cases}
1 \text{ or } 2& \text{if } i=j\\
0 & \text{ if } i\neq j\end{cases}\] and that if $\alpha_i\cap \beta_i=\{p_i^-,p_i^+\}$, then $p_i^-$ and $p_i^+$ differ by Maslov grading one\footnote{More precisely we mean that $\ve{x}\times p_i^+$ and $\ve{x}\times p_i^-$ differ by Maslov grading 1 for any $\ve{x}$.}. We do not assume that the $\alpha_i$ curves are small isotopies of the $\beta_j$ curves. Let $\theta^+=p_1^+\times \cdots \times p_{n-1}^+$ denote the top graded (partial) intersection point. Assume that $|\alpha_n\cap \beta_n|=2$ and write $p_n^-$ and $p_n^+$ for the two points of $\alpha_n\cap \beta_n$.

\begin{define}\label{def:stronglypositive}Under the assumptions of the previous paragraph, we say that $(\Sigma, \ve{\alpha},\ve{\beta},\ve{w})$ is \textbf{strongly positive} with respect to $p_n^+$ if for every nonnegative disk $\phi\in \pi_2(\theta^+\times p_n^+,\ve{y}\times p_n^+)$ we have that
\[(\mu-(m_1+m_2))(\phi)=(\mu-(n_1+n_2))(\phi)\ge 0,\] with equality to zero iff $\phi$ is the constant disk. Here $m_1,m_2,n_1,n_2$ denote the multiplicities adjacent to the point $p_n^+$, appearing in the following counterclockwise order: $n_1,m_1,n_2,$ then $m_2$.
\end{define}

Note that  $m_1+m_2=n_1+n_2$ for any disk $\phi\in \pi_2(\ve{x}\times p_n^+, \ve{y}\times p_n^+)$, by the vertex relations.

\begin{figure}[ht!]
\centering
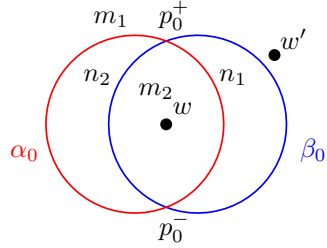
\caption{The diagram $\cH_0=(S^2,\alpha_0,\beta_0,w,w')$ in Lemma \ref{lem:stronglypositive} which is strongly positive at $p_0^+$, as well as the multiplicities $m_1,n_1,m_2$ and $n_2$.\label{fig::48}}
\end{figure}

We now describe a class of diagrams which are strongly positive, which will be sufficient for our purposes:

\begin{lem}\label{lem:stronglypositive}The diagram $\cH_0=(S^2,\alpha_0,\beta_0,w,w')$ in Figure \ref{fig::48} is strongly positive with respect to $p_0^+$, the intersection point of $\alpha_0$ and $\beta_0$ with higher relative grading. If $\cH=(\Sigma, \ve{\alpha},\ve{\beta},\ve{w})$ is a diagram with a distinguished intersection point $p_n^+\in \alpha_n\cap \beta_n$ where $|\alpha_n\cap \beta_n|=2$, and $\cH'=(\Sigma',\ve{\alpha}',\ve{\beta}',\ve{w}')$ is the result of any of the following moves, then $\cH'$ is strongly positive with respect to $p_n^+$ iff $\cH$ is strongly positive with respect to $p_n^+$:
\begin{enumerate}
\item \label{list:strongmove1} $(1,2)$-stabilization\footnote{If $\cH=(\Sigma, \ve{\alpha},\ve{\beta},\ve{w})$ is a Heegaard diagram for $Y$, a $(1,2)$-stabilization of $\cH$ is obtained by taking an embedded torus $\bT^2$ inside of a 3-ball in $Y\setminus \Sigma$, together with curves $\alpha_0$ and $\beta_0$ with $|\alpha_0\cap \beta_0|=1$, which bound compressing disks with boundary on $\Sigma$, and letting $\cH'$ be the diagram $(\Sigma\# \bT^2,\ve{\alpha}\cup \{\alpha_0\}, \ve{\beta}\cup \{\beta_0\}, \ve{w})$, where $\Sigma\# \bT^2$ is the internal connected sum.};
\item \label{list:strongmove2}taking the disjoint union of $\cH$ with the standard diagram $(\bT^2,\alpha_0,\beta_0,w)$ for $(S^3,w)$.
\item \label{list:strongmove3}performing surgery on an embedded 0-sphere $\{q_1,q_2\}\subset \Sigma\setminus (\ve{\alpha}\cup \ve{\beta})$ by removing small disks from $\Sigma$, and connecting the resulting boundary components with an annulus with new $\alpha_0$ and $\beta_0$ curves with $|\alpha_0\cap \beta_0|=2$, and which are isotopic to each other, and homotopically nontrivial in the annulus;
\end{enumerate}
\end{lem}
\begin{proof} We first note that $\cH_0$ is strongly positive with respect to $p_0^+$, because the Maslov index of any disk is given by
\[\mu(\phi)=(m_1+m_2+n_1+n_2)(\phi),\] by Lemma \ref{lem:Maslovindexdisk}, where $m_1,m_2,n_1,n_2$ are multiplicities appearing in the counterclockwise order $m_1,n_1,m_2,n_2$ around $p_0^+$, as in Figure \ref{fig::48}. Hence for any disk $\phi$, we have
\[(\mu-(n_1+n_2))(\phi)=(m_1+m_2)(\phi),\] which is certainly nonnegative. For a disk $\phi\in \pi_2(p_0^+,p_0^+)$ we also have $(m_1+m_2)(\phi)=(n_1+n_2)(\phi)$, so the above quantity is positive iff $\phi$ is has positive multiplicities.

We now address moves (1)-(3).

We first consider move \eqref{list:strongmove1}. If $\cH$ is a diagram, and $\cH'$ is the result of $(1,2)$-stabilization, we note that there is an isomorphism
\[ \sigma_*:\pi_2^{\cH}(\ve{x}\times p_n^+,\ve{y}\times p_n^+)\to \pi_2^{\cH'}(\ve{x}\times c\times p_n^+, \ve{y}\times c\times p_n^+),\] where $c$ is the intersection of the new $\ve{\alpha}$- and $\ve{\beta}$-curves. Furthermore we note that 
\[(\mu-(n_1+n_2))(\phi)=(\mu-(n_1+n_2))(\sigma_*\phi),\]
 from which the claim follows easily.

We now consider move \eqref{list:strongmove2}. Suppose $\cH'$ is formed from $\cH$ by taking the disjoint union of $\cH$ with a diagram 
$(\bT^2,\alpha_0,\beta_0,w)$. Homology classes on $\cH'$ are of the form $\phi\sqcup k\cdot [\bT^2]$, where $\phi$ is a homology disk on $\cH$. One has
\[\mu(\phi\sqcup k\cdot[\bT^2])=\mu(\phi)+2k,\] from which the claim follows easily.

Finally we consider move \eqref{list:strongmove3}, corresponding to surgering on an embedded 0-sphere $\{q_1,q_2\}\subset \Sigma\setminus (\ve{\alpha}\cup \ve{\beta}\cup \ve{w})$. Write $\alpha_0$ and $\beta_0$ for the new curves on the annulus, and $\{\theta_0^+,\theta_0^-\}=\alpha_0\cap \beta_0$. Suppose that $y\in \alpha_0\cap \beta_0$ is a choice of intersection point. We can define an injection
\[\iota_{y}:\pi_2^{\cH}(\theta^+\times p_n^+,\ve{y}\times p_n^+)\to \pi_2^{\cH'}(\theta^+\times \theta_0^+\times p_n^+, \ve{y}\times y\times p_n^+) .\] For $y=\theta_0^+$, we define $\iota_{\theta^+_0}(\phi)$ to be the disk on the surgered diagram which has no change across the curve $\beta_0$, but which agrees with the disk $\phi$ away from the $\alpha_0$ and $\beta_0$ curves. For $y=\theta_0^-$, a map $\iota_{\theta_0^-}$ can be defined by defining it to be the the map $\iota_{\theta_0^+}$, defined above, composed with the map on disks obtained by splicing in a choice of one of the thin strips from $\theta_0^+$ to $\theta_0^-$. An easy computation shows that
\[(\mu-(m_1+m_2))(\iota_{\theta_0^+}(\phi))=(\mu-(m_1+m_2))(\phi)\] while
\[(\mu-(m_1+m_2))(\iota_{\theta_0^-}(\phi))=(\mu-(m_1+m_2))(\phi)+1.\] Any disk in $\pi_2(\theta^+\times \theta_0^+\times p_n^+, \ve{y}\times y\times p_n^+)$ is equal to one which is in the image of $\iota_y$, with $n\cdot \cP$ spliced in, where $\cP$ is the periodic domain which is $+1$ in one of the small strips between $\alpha_0$ and $\beta_0$, and $-1$ in the other. We note that 
\[\mu(\cP)=m_1(\cP)=m_2(\cP)=0.\] From these observations it follows easily that $\cH$ is strongly positive with respect to $p_n^+$ iff $\cH'$ is.
\end{proof}

\begin{rem}\label{rem:stronglypositivediagrams}Suppose $\cH=(\Sigma, \ve{\alpha}',\ve{\alpha},\ve{w})$ is obtained by taking attaching curves $\ve{\alpha}$ and letting $\ve{\alpha}'$ be small Hamiltonian isotopies of the curves in $\ve{\alpha}$. Let $\alpha_s$ be a new curve in $\Sigma\setminus \ve{\alpha}$ which doesn't intersect any $\ve{\alpha}'$-curves. If $\alpha_s'$ is the result of handlesliding $\alpha_s$ across a curve in $\ve{\alpha}$, then $(\Sigma, \ve{\alpha}'\cup \{\alpha_s'\}, \ve{\alpha}\cup \{\alpha_s\},\ve{w}\cup \{w\})$ is strongly positive at $p^+$, the intersection point of $\alpha_s'\cap \alpha_s$ with higher grading. Here $w$ is a new basepoint  in one of the regions adjacent to $p^+$.

Similarly if $\cH=(\Sigma, \ve{\alpha}',\ve{\alpha},\ve{w})$ is the result of handlesliding a curve in $\ve{\alpha}$ across another curve in $\ve{\alpha}$, and $\alpha_s$ and $\alpha_s'$ are two new curves which are Hamiltonian isotopies of each other, then $(\Sigma, \ve{\alpha}'\cup \{\alpha_s'\},\ve{\alpha}\cup \{\alpha_s\},\ve{w}\cup \{w\})$ is strongly positive with respect to the intersection point of $\alpha_s'\cap \alpha_s$ of higher relative grading.

Note that a diagram $(\Sigma_g, \ve{\alpha}',\ve{\alpha},w)$ where the curves in $\ve{\alpha}'$ are small isotopies of the curves in $\ve{\alpha}$ with $g(\Sigma_g)=|\ve{\alpha}'|=|\ve{\alpha}|=g$, is not a strongly positive diagram at any point, since $[\Sigma_g]$ represents a positive homology class in $\pi_2(\theta^+,\theta^+)$ with $\mu([\Sigma_g])-m_1-m_2=2-1-1=0$. Strongly positive diagrams always have multiple basepoints. The prototypical example is the one resulting from handlesliding a quasi-stabilization curve $\alpha_s$ across an $\ve{\alpha}$-curve, as in Figure \ref{fig::3}.
\end{rem}

\begin{figure}[ht!]
\centering
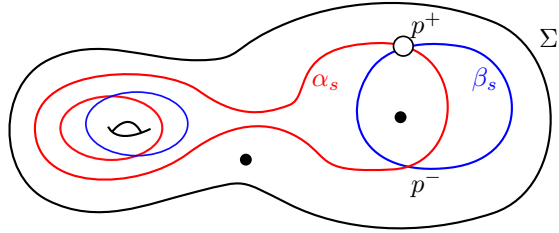
\caption{A diagram $(\Sigma, \ve{\alpha},\ve{\beta})$ with two basepoints which is strongly positive with respect to the point $p^+$. The curves $\alpha_s$ and $\beta_s$ are curves on which one could perform the quasi-stabilization operation of triangles in Theorem \ref{thm:doublequasitriangle}. \label{fig::3}}
\end{figure}

\subsection{Quasi-stabilizing Heegaard triples along two curves}

We now consider the effect on the triangle maps of quasi-stabilizating along two curves. Our analysis follows a similar spirit to the proof of \cite{MOIntSurg}*{Prop. 5.2}. Suppose that $(\Sigma, \ve{\alpha},\ve{\beta},\ve{\gamma},\ve{w},\ve{z})$ is a Heegaard triple with a distinguished point $p^+\in \Sigma\setminus (\ve{\alpha}\cup \ve{\beta}\cup \ve{\gamma}\cup \ve{w}\cup \ve{z})$. Suppose also that $\alpha_s$ and $\beta_s$ are two choices of curves in $\Sigma\setminus (\ve{\alpha}\cup \ve{w}\cup \ve{z})$ and $(\Sigma\setminus (\ve{\beta}\cup \ve{w}\cup \ve{z})$ respectively, which intersect only at $p^+$, and another point $p^-\in \Sigma$. We can form the diagram $\bar{\cT}_{\alpha_s,\beta_s,p^+}$, obtained by quasi-stabilizing along both $\alpha_s$ and $\beta_s$, simultaneously, at the point $p^+$. This corresponds to removing a small disk containing $p^+$, and inserting the diagram shown in Figure \ref{fig::46}, with a disk centered around $p_0^-$ removed.

\begin{figure}[ht!]
\centering
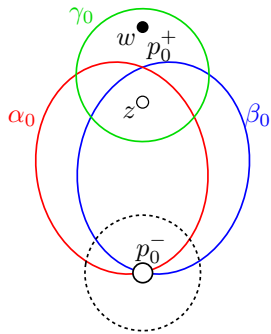
\caption{The diagram we insert into a Heegaard triple diagram $\cT$ along the curves $\alpha_s$ and $\beta_s$ to form the diagarm $\bar{\cT}_{\alpha_s,\beta_s,p^+}$. We cut out the solid circle marked with $p_0^-$ and stretch the almost complex structure along the dashed circle. \label{fig::46}}
\end{figure}

\begin{thm}\label{thm:doublequasitriangle}Suppose that $\cT=(\Sigma, \ve{\alpha},\ve{\beta},\ve{\gamma},\ve{w},\ve{z})$ is a Heegaard triple with curves $\alpha_s$ and $\beta_s$, intersecting at two points $p^+$ and $p^-$ and let $\bar{\cT}_{\alpha_s,\beta_s,p^+}$ be the Heegaard triple result from quasi-stabilization, as described above. If $(\Sigma,\ve{\alpha}\cup \{\alpha_s\},\ve{\beta}\cup \{\beta_s\},\ve{w}\cup \{w\})$ is a strongly positive diagram for $(S^1\times S^2)^{\# k}$ with respect to the point $p^+$ (Definition \ref{def:stronglypositive}), and $\cJ$ is gluing data for stretching along the dashed circle, then for sufficiently large $T$, with respect to the almost complex structure $\cJ(T)$, there are identifications
\[F_{\cT^+_{\alpha_s,\beta_s,p^+},\bar{\frs},\cJ(T)}(\Theta^+_{\alpha\beta}\times p_0^+,\ve{x}\times \cdot)=\begin{pmatrix}F_{\cT,\frs,J_s}(\Theta_{\alpha\beta}^+,\ve{x})&0\\0& F_{\cT,\frs,J_s}(\Theta_{\alpha\beta}^+,\ve{x})
\end{pmatrix},\] where $\cdot$ denotes $x^\pm\in \beta_0\cap \gamma_0$ and the matrix on the right denotes the matrix decomposition of the map based on the decompositions given $x^{\pm}$ and $y^{\pm}$.
\end{thm}

As usual, the argument proceeds by a Maslov index calculation, which we use to put constraints on the homology classes of holomorphic curves which can appear in a weak limit as we let the parameter $T$ approach $+\infty$. Once we determine which homology classes of triangles can appear, we can use gluing results to explicitly count holomorphic curves.

\begin{figure}[ht!]
\centering
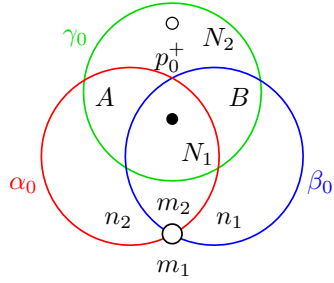
\caption{Multiplicities for a triangle on the diagram $(S^2,\alpha_0,\beta_0,\gamma_0,w,z)$.\label{fig::47}}
\end{figure}

\begin{lem}\label{lem:Maslovindextrianglediagram} Suppose $\psi\in \pi_2(x,y,z)$ is a homology disk on $(S^2,\alpha_0,\beta_0,\gamma_0,w,z)$, shown in Figure \ref{fig::47}. Then
\[\mu(\psi)=n_1+n_2+N_1+N_2.\]
\end{lem}
\begin{proof}The formula is easily checked for any of the Maslov index zero small triangles, and respects splicing in any Maslov index 1 strip. Since any two triangles on this diagram differ by splicing in some number of the Maslov index 1 strips, the formula follows in full generality.
\end{proof}

We can now prove Theorem \ref{thm:doublequasitriangle}.

\begin{proof}Suppose that $u_i$ is a sequence of holomorphic triangles of Maslov index zero representing a class $\psi\in\pi_2(\Theta_{\alpha\beta}^+\times p_0^+,\ve{x}\times x, \ve{y}\times y)$, for almost complex structure $\cJ(T_i)$, where $T_i$ is a sequence of neck lengths approaching $+\infty$. Adapting the proof of Proposition \ref{lem:differentialcomp}, the limiting curves which appear can be arranged into three classes of broken holomorphic curves:
\begin{enumerate}
\item a broken holomorphic triangle $u_\Sigma$ representing a homology class $\psi_\Sigma$ on $(\Sigma, \ve{\alpha},\ve{\beta},\ve{\gamma})$ which has no boundary components on $\alpha_s$ or $\beta_s$;
\item a broken holomorphic disk $u_{\alpha\beta}$ on $(\Sigma, \ve{\alpha}\cup\{\alpha_s\},\ve{\beta}\cup \{\beta_s\})$ representing a class $\phi_{\alpha\beta}\in \pi_2(\Theta_{\alpha\beta}^+\times p^+, \ve{y}\times p^+)$, for some $\ve{y}\in \bT_\alpha\cap \bT_\gamma$;
\item a broken holomorphic triangle $u_0$ on $(S^2,\alpha_0,\beta_0,\gamma_0)$ which represents a homology triangle $\psi_0\in \pi_2(p_0^+,x,y)$.
\end{enumerate}

We now wish to write down the Maslov index of $\psi$ in terms of the Maslov indices and multiplicities of $\psi_\Sigma, \psi_0$ and $\phi_{\alpha\beta}$. Let $m_1(\cdot), m_2(\cdot),n_1(\cdot),$ and $n_2(\cdot)$ denote the the multiplicities of a homology curve in the regions surrounding $p^+$ or $p_0^-$, as in Figure \ref{fig::47}. In \cite{SMaslov}, Sarkar derives a formula for the Maslov index of a homology triangle $\rho\in \pi_2(\ve{x},\ve{y},\ve{z})$ which can be computed entirely from the domain $\cD(\rho)$. Writing $\cD=\cD(\rho)$, the formula reads
\[\mu(\rho)=e(\cD)+n_{\ve{x}}(\cD)+n_{\ve{y}}(\cD)+a(\cD). c(\cD)-\frac{d}{2},\] where $d=|\ve{\alpha}|=|\ve{\beta}|=|\ve{\gamma}|$. Here $a(\cD)$ is defined intersection $\d \cD\cap \ve{\alpha}$ (viewed as a 1-chain), and $c(\cD)$ is defined similarly, using the $\ve{\gamma}$-curves. The quantity $a(\cD).c(\cD)$ is defined as the average of the four algebraic intersection numbers of $a'(\cD)$ and $c(\cD)$, where $a'(\cD)$ is a translate of $a(\cD)$ in any of the four ``diagonal directions''. If $s\in\alpha_i\cap \beta_j$, then $n_s(\cD)$ is the average of the multiplicities in the regions surrounding $s$, and if $\ve{s}$ is a set of such intersection points, then $n_{\ve{s}}(\cD)$ is the sum of the $n_s(\cD)$ ranging over $s\in \ve{s}$. 

For a homology triangle $\psi\in \pi_2(\Theta_{\alpha\beta}^+\times p_0^+,\ve{x}\times x, \ve{y}\times y)$ which can be decomposed into homology classes $\psi_{\Sigma},\phi_{\alpha\beta},$ and $\psi_0$ as above (as any homology class admitting holomorphic representatives for arbitrarily large neck length can) we observe that $\mu(\psi)$ can be computed by adding up $\mu(\psi_\Sigma),\mu(\phi_{\alpha\beta})$ and $\mu(\psi_0)$, then subtracting the quantities which are over-counted. This corresponds to subtracting 
$\tfrac{1}{2}(m_1+m_2+n_1+n_2)(\psi),$ which is the excess of Euler measure resulting from removing balls centered at $p^+$ and at $p_0^-$ (note that the Euler measure of a quarter disk is $\tfrac{1}{4}$), and subtracting 
\[2n_{p^+}(\phi_{\alpha\beta})=\tfrac{1}{2}(n_1+n_2+m_1+m_2)(\phi_{\alpha\beta}),\] which is the quantity in the expression 
\[\mu(\phi_{\alpha\beta})=e(\cD(\phi_{\alpha\beta}))+n_{\Theta^+_{\alpha\beta}\times p^+}(\phi_{\alpha\beta})+n_{\Theta^+_{\alpha\beta}\times p^+}(\phi_{\alpha\beta})\]
\[=e(\cD(\phi_{\alpha\beta}))+n_{\Theta_{\alpha\beta}^+}(\phi_{\alpha\beta})+n_{\ve{y}}(\phi_{\alpha\beta})+\tfrac{1}{2}(n_1+n_2+m_1+m_2)(\phi_{\alpha\beta})\] which does not contribute to $\mu(\psi)$. Adding these contributions, we get
\begin{equation}\mu(\psi)=\mu(\psi_\Sigma)+\mu(\psi_0)+\mu(\phi_{\alpha\beta})-\tfrac{1}{2}(n_1+n_2+m_1+m_2)(\psi)-\tfrac{1}{2}(n_1+n_2+m_1+m_2)(\phi_{\alpha\beta}).\label{eq:Maslovindex1}
\end{equation} Writing $\psi_0\in \pi_2(p_0^+,x,y)$, using the vertex multiplicity relations around $p_0^-$, it is an easy computation that
\[(m_1+m_2)(\psi)=(n_1+n_2)(\psi).\] Note also that $m_i(\psi_0)=m_i(\psi)$ and similarly for the multiplicities $n_i$, since we are grouping all holomorphic curves on $(S^2,\alpha_0,\beta_0,\gamma_0)$ appearing in the weak limit into the homology class $\psi_0$. Using the Maslov index formula from Lemma \ref{lem:Maslovindextrianglediagram} for $\psi_0$,  we get from Equation \eqref{eq:Maslovindex1} that 
\[\mu(\psi)=\mu(\psi_\Sigma)+\mu(\phi_{\alpha\beta})+(m_1+m_2+N_1+N_2)(\psi_0)\]\[-\tfrac{1}{2}(m_1+m_2+n_1+n_2)(\psi_0)-\tfrac{1}{2}(n_1+n_2+m_1+m_2)(\phi_{\alpha\beta})\] which  reduces to

\begin{equation}\mu(\psi)=\mu(\psi_\Sigma)+(N_1+N_2)(\psi_0)+\mu(\phi_{\alpha\beta})-\tfrac{1}{2}(n_1+n_2+m_1+m_2)(\phi_{\alpha\beta}),\label{eq:Maslovindex2}\end{equation}
 since $\psi_0$ does not have $p^+$ as a vertex, so the vertex relations at $p^+$ yield
\[\tfrac{1}{2}(n_1+n_2+m_1+m_2)(\psi_0)\]\[=(n_1+n_2)(\psi_0)=(m_1+m_2)(\psi_0).\]

We now use the assumption that $(\Sigma,\ve{\alpha}\cup\{\alpha_s\},\ve{\beta}\cup \{\beta_s\},\ve{w})$ is strongly positive with respect to $p^+$, and hence
\[\mu(\phi_{\alpha\beta})-\tfrac{1}{2}(n_1+n_2+m_1+m_2)(\phi_{\alpha\beta})\ge 0,\] with equality iff $\phi_{\alpha\beta}$ is a constant disk. We note that $\mu(\psi_\Sigma)\ge 0$, since $\psi_\Sigma$ admits a broken holomorphic representative. Hence the formula for $\mu(\psi)$ in Equation \eqref{eq:Maslovindex2} can be written as a sum of nonnegative expressions, and hence each must be zero if $\mu(\psi)=0$. Thus
\[0=\mu(\psi)=\mu(\psi_\Sigma)=\mu(\phi_{\alpha\beta})=N_1=N_2.\]

Note first that this implies that $\psi_\Sigma$ is a constant disk, since generically there are no nonconstant, Maslov index zero disks. From here the argument proceeds in a familiar manner. Note that $\psi_\Sigma$ is a Maslov index 0 broken holomorphic triangle, and hence must be a genuine holomorphic triangle, as nonconstant holomorphic disks, boundary degenerations, and closed surfaces all have strictly positive Maslov index. As such, since $\psi_\Sigma$ is represented by a holomorphic triangle with no boundary components mapped to $\alpha_s$ or $\beta_s$, we must have $m_1(\psi)=m_2(\psi)=n_1(\psi)=n_2(\psi)$.

We now claim that this implies that the off diagonal entries of the matrix representing the triangle map in the statement are zero. Writing $\psi_0\in \pi_2(p_0^+, x,y)$, and using the multiplicities in Figure \ref{fig::47}, the vertex relations at $p_0^+$ read
\[A+B=1,\] and hence exactly one of $A$ and $B$ is 1, and the other is 0, since $A,B\ge 0$. Since $n_1=n_2=m_1=m_2$, we know that by subtracting some number of copies of the $\gamma_0$-boundary degeneration of Maslov index 2 with $N_1=N_2=0$, we get a homology triangle in $\pi_2(p_0^+,x,y)$ with $N_1,N_2,m_1,m_2,n_1,$ and $n_2$ all zero. There are two only homology triangles satisfying that condition, one is in $\pi_2(p_0^+,x^+,y^+)$ and the other is in $\pi_2(p_0^+,x^-,y^-)$, implying that $\psi_0$ itself must be in one of those sets. Hence the off diagonal entries of the matrix are zero.

Since $u_\Sigma$ is a genuine Maslov index zero holomorphic triangle, there must be a curve in $u_0'$ in the broken holomorphic triangle $u_0$ which matches, i.e. which satisfies 
\[\rho^{p^+}(u_\Sigma)=\rho^{p_0}(u_0').\]Recall that if $u:S\to \Sigma\times \Delta$ is a holomorphic map and $q\in \Sigma$ is a point, we define 
\[\rho^q(u)=(\pi_\Delta\circ u)(\pi_\Sigma \circ u)^{-1}(q)\in \Sym^{n_q(u)}(\Delta).\] Since this in particular forces $n_{p^+}(u_\Sigma)=n_{p_0^-}(u_0')$, it is easy to see that there can be no other curves in the broken curve $u_0$ since there are no multiplicities on $(S^2,\alpha_0,\beta_0,\gamma_0)$ which could be increased without increasing $n_1,m_1,n_2,m_2,N_1,$ and $N_2$ while still preserving the vertex relations. By standard gluing arguments (cf. \cite{LCylindrical}*{App. A}, \cite{HFplusTQFT}*{Sec. 6.4}) the count of prematched triangles\footnote{Recall that a prematched triangle is a pair $(u_\Sigma, u_0)$ where $u_\Sigma$ and $u_0$ are holomorphic triangles representing $\psi_\Sigma$ and $\psi_0$ on $(\Sigma,\ve{\alpha},\ve{\beta},\ve{\gamma})$ and $(S^2,\alpha_0,\beta_0,\gamma_0)$ respectively and $\rho^{p^+}(u_\Sigma)=\rho^{p_0}(u_0')$.} is equal to the count holomorphic triangles in $\# \cM_{\cJ(T)}(\psi_\Sigma\# \psi_0)$ for sufficiently large $T$.

 For $x$ and $y$ of the same relative grading (both the top intersection points or both the lower intersection points), for each $k$ there is a unique homology class $\psi_k$ on $(S^2,\alpha_0,\beta_0,\gamma_0)$ in $\pi_2(p_0^+,x,y)$ with $n_1=n_2=m_1=m_2=k$. Thus, adapting the proof of Proposition \ref{lem:differentialcomp}, it is sufficient to count holomorphic triangles on $(S^2,\alpha_0,\beta_0,\gamma_0)$ of homology class $\psi_k$ which match a fixed divisor $\ve{d}\in \Sym^k(\Delta)$. The count of such holomorphic triangles matching $\ve{d}$ is generically $1\pmod{2}$, as can be seen from a Gromov compactness argument nearly identical to the one done at the end of Proposition \ref{lem:differentialcomp}. Hence the diagonal entries of the triangle map matrix are as claimed, completing the proof.
\end{proof}

\section{Invariance of the quasi-stabilization maps}
\label{sec:quasistabilizationnatural}
In this section, we combine the results of the previous section to prove invariance of the quasi-stabilization maps:

\begin{customthm}{\ref{thm:A}} Assume that $(\sigma,\frP)$ is a coloring of the basepoints $\ve{w}\cup \ve{z}$ which is extended by the coloring $(\sigma',\frP)$ of the basepoints $\ve{w}\cup \ve{z}\cup \{w,z\}$. The quasi-stabilization operation induces well defined maps
\[S_{w,z}^+:CFL_{UV}^\infty(Y,L,\ve{w},\ve{z},\sigma,\frP,\frs)\to CFL_{UV}^\infty(Y,L,\ve{w}\cup \{w\},\ve{z}\cup \{z\},\sigma',\frP,\frs)\] and
\[S_{w,z}^-:CFL_{UV}^\infty(Y,L,\ve{w}\cup \{w\},\ve{z}\cup \{z\},\sigma',\frP,\frs)\to CFL_{UV}^\infty(Y,L,\ve{w},\ve{z},\sigma,\frP,\frs),\] which are well defined invariants up to $\frP$-filtered $\Z_2[U_\frP]$-chain homotopy.
\end{customthm}

The proof is to construct the maps for choices of Heegaard diagram and auxiliary data, and show that the maps we describe are independent from that auxiliary data and the choice of diagram.

 If $\cH=(\Sigma, \ve{\alpha},\ve{\beta},\ve{w},\ve{z})$ is a diagram for $\bL=(L,\ve{w},\ve{z})$,  recall from Subsection \ref{subsec:topprelim} that if $A$ denotes the component of $\Sigma\setminus \ve{\alpha}$ containing the basepoints of $(L,\ve{w},\ve{z})$ adjacent to $w$ and $z$, then we pick a point $p\in A\setminus (\ve{\alpha}\cup \ve{\beta}\cup \ve{w}\cup \ve{z})$ and a simple closed curve $\alpha_s\subset A\setminus \ve{\alpha}$ to form a diagram $\bar{\cH}_{p,\alpha_s}$. Let $\cJ$ denote gluing data (cf. Subsection \ref{subsec:gluingdata}) for performing the special connected sum operation at $p\in \Sigma$ and $p_0\in S^2$ and gluing almost complex structures on $\cH$ and $\cH_0$ together.

We now define maps \[S^{+}_{w,z,\cH,p,\alpha_s,\cJ,T}:CFL^\infty_{UV,J_s}(\cH, \frs)\to CFL_{UV,\cJ(T)}^\infty(\bar{\cH}_{p,\alpha_s}, \frs)\] and
\[S^-_{w,z,\cH,p,\alpha_s,\cJ,T}:CFL_{UV,\cJ(T)}^\infty(\bar{\cH}_{p,\alpha_s}, \frs)\to CFL^\infty_{UV,J_s}(\cH, \frs)\] by the formulas
\[S^{+}_{w,z,\cH,p,\alpha_s,\cJ}(\ve{x})=\ve{x}\times \theta^+\] and
\[S^{-}_{w,z,\cH,p,\alpha_s,\cJ}(\ve{x}\times \theta^+)=0 \qquad \text{ and }\qquad  S^{+}_{w,z,\cH,p,\alpha_s,\cJ}(\ve{x}\times \theta^-)=\ve{x},\] where $J_s$ denotes the almost complex structure on $\Sigma$ associated to the gluing data $\cJ$ and $T$ is sufficiently large.

The maps $S_{w,z,\cH,p,\alpha_s,\cJ,T}^{\pm}$ can be extended to the entire $\frP$-filtered chain homotopy-type invariant by pre- and post-composing with change of diagram maps and change of almost complex structure maps. By functoriality of the change of diagrams maps, we get well defined maps $S_{w,z,\cH,p,\alpha_s,\cJ,T}^{\pm}$ between $CFL_{UV}^\infty(Y,L,\ve{w},\ve{z},\sigma,\frP,\frs)$ and $CFL_{UV}^\infty(Y,L,\ve{w}\cup \{w\},\ve{z}\cup \{z\},\sigma',\frP,\frs)$, though of course we need to show independence from the choices of $\cH,p,\alpha_s,\cJ,$ and $T$.

We first address independence from $\cJ$ and the parameter $T$. Recalling Lemma \ref{lem:transitionmapsstabilize}, there is an $N$ such that if $T,T'> N$ we have
\[\Phi_{\cJ(T)\to \cJ(T')}\simeq \begin{pmatrix}1& *\\
0& 1\end{pmatrix}.\] Hence, as with the free stabilization maps defined in \cite{HFplusTQFT}, we define the maps $S_{w,z,\cH,p,\alpha_s,\cJ}^{\pm}$ to be between the complexes $CFL_{J_s,UV}^\infty(\cH,\frs)$ and $CFL_{\cJ(T),UV}^\infty(\bar{\cH}_{p,\alpha_s},\frs)$ for any $T$ greater than any such $N$.

\begin{lem}The maps $S_{w,z,\cH,p,\alpha_s,\cJ}^{\pm}$ are independent of $\cJ$ and the parameter $T$.
\end{lem}

\begin{proof}  Lemma \ref{lem:transitionmapsstabilize} implies that $S_{w,z,\cH,p,\alpha_s,\cJ,T}^{\pm}$ is independent of $T$ for any $T$ which is sufficiently large. Let $S_{w,z,\cH,p,\alpha_s,\cJ}^{\pm}$ denote the common map. Lemma \ref{lem:changeofalmostcomplexstructure} implies that the maps $S^{\pm}_{w,z,\cH,p,\alpha_s,\cJ}$ are independent of $\cJ$. We denote the common map from now on by $S_{w,z,\cH,p,\alpha_s}^{\pm}$.\end{proof}

\begin{lem}\label{lem:quasistabindependentofalphas}For a fixed diagram $\cH$ with fixed $p\in \Sigma$, the maps $S_{w,z,\cH,p,\alpha_s}^{\pm}$ are independent of $\alpha_s$.
\end{lem}

\begin{proof} This follows from the triangle map computation in Theorem \ref{thm:doublequasitriangle}, which allows one to change the $\alpha_s$ curve through  a sequence of isotopies of the $\alpha_s$ curve, and handleslides of the $\alpha_s$ curve over other $\ve{\alpha}$-curves, each of which can be realized by quasi-stabilizing a Heegaard triple $(\Sigma, \ve{\alpha}',\ve{\alpha},\ve{\beta},\ve{w},\ve{z})$ along two curves $\alpha_s'$ and $\alpha_s$ with $\alpha_s'\cap \alpha_s=\{p^+,p^-\}$, such that $(\Sigma, \ve{\alpha}'\cup \{\alpha_s'\},\ve{\alpha}\cup \{\alpha_s\},\ve{w})$ is strongly positive at $p^+$ (cf. Remark \ref{rem:stronglypositivediagrams}). For each such move, by Theorem \ref{thm:doublequasitriangle} the change of diagrams map can be written as
\[\Phi^{\ve{\alpha}\cup \{\alpha_s\}\to \ve{\alpha}'\cup \{\alpha_s'\}}_{\ve{\beta}\cup \{\beta_s\}}=\begin{pmatrix}\Phi^{\ve{\alpha}\to \ve{\alpha}'}_{\ve{\beta}}&*\\
0&\Phi^{\ve{\alpha}\to \ve{\alpha}'}_{\ve{\beta}}
\end{pmatrix},\] where the curves in $\ve{\alpha}'$ are small Hamiltonian isotopies of the curves in $\ve{\alpha}$. Similarly by Theorem \ref{thm:doublequasitriangle} we also have
\[\Phi^{\ve{\alpha}\cup \{\alpha_s\}\to \ve{\alpha}'\cup \{\alpha_s\}}_{\ve{\beta}\cup \{\beta_s\}}=\begin{pmatrix}\Phi^{\ve{\alpha}\to \ve{\alpha}'}_{\ve{\beta}}&*\\
0&\Phi^{\ve{\alpha}\to \ve{\alpha}'}_{\ve{\beta}}
\end{pmatrix},\] completing the proof.
\end{proof}

We now let $S_{w,z,\cH,p}^{\pm}$ denote the map $S_{w,z,\cH,p,\alpha_s}^{\pm}$ for any choice of $\alpha_s$. We now prove independence from the choice of point $p\in \Sigma$.

\begin{lem}\label{lem:independentofp}Given a fixed diagram $\cH$, the maps $S_{w,z,\cH,p}^{\pm}$ are independent of the choice of point $p\in \Sigma$.
\end{lem}
\begin{proof} Let $A$ denote the component of $\Sigma\setminus \ve{\alpha}$ containing the basepoints on $\bL$ which are adjacent to $w$ and $z$. Let  $p$ and $p'$ be two choices of points in $A\setminus (\ve{\beta}\cup \ve{w}\cup \ve{z})$. Let $\phi_t$ be an isotopy $\phi_t:\Sigma\to \Sigma$ which fixes $\Sigma\setminus A$ and maps $p$ to $p'$. Recall that the surfaces $\bar{\Sigma}_p$ were well defined up to an isotopy fixing $\ve{\alpha}\cup \ve{\beta}\cup \ve{w}\cup \ve{z}$ and mapping $L$ to $L$. Extend $\phi=\phi_1$ to an isotopy of $Y$ which fixes $(\Sigma\setminus A)\cup \ve{w}\cup \ve{z}\cup \{w,z\}$ and maps $L$ to $L$. By definition
\[(\phi)_*(\bar{\Sigma}_p)=\bar{\Sigma}_{p'},\] as embedded surfaces. The diffeomorphism $\phi$ fixes all the curves in $\ve{\alpha}$, but moves some of the $\ve{\beta}$-curves which pass through the region $A$.

 Observe the following factorizations: 
\begin{equation}\Phi_{(\Sigma, \ve{\alpha},\ve{\beta},J_s)\to (\Sigma, \ve{\alpha},\ve{\beta},\phi_*J_s)}\simeq \Phi_{\phi_*\ve{\beta}\to\ve{\beta}}^{\ve{\alpha}}\circ \phi_*\label{eq:factortransitionmaps}
\end{equation} and similarly
\[\Phi_{(\bar{\Sigma}_p, \ve{\alpha}\cup \{\alpha_s\},\ve{\beta}\cup \{\beta_0\},\cJ(T))\to (\bar{\Sigma}_{p'}, \ve{\alpha}\cup \{\phi_*\alpha_s\},\ve{\beta}\cup \{\phi_* \beta_0\},\phi_*\cJ(T))}\simeq \Phi_{(\phi_*\ve{\beta})\cup \{\phi_* \beta_0\}\to\ve{\beta}\cup \{\phi_*\beta_0\}}^{\ve{\alpha}\cup \{\phi_* \alpha_s\}}\circ \phi_*.\] Using Theorem \ref{thm:doublequasitriangle}, for sufficiently stretched almost complex structure we can write
\[\Phi_{(\phi_*\ve{\beta})\cup \{\phi_* \beta_0\}\to\ve{\beta}\cup \{\phi_*\beta_0\}}^{\ve{\alpha}\cup \{\phi_* \alpha_s\}}\simeq \begin{pmatrix}\Phi_{\phi_*\ve{\beta}\to\ve{\beta}}^{\ve{\alpha}}&*\\
0& \Phi_{\phi_*\ve{\beta}\to\ve{\beta}}^{\ve{\alpha}}
\end{pmatrix},\] and hence
\[\Phi_{(\phi_*\ve{\beta})\cup \{\phi_* \beta_0\}\to\ve{\beta}\cup \{\phi_*\beta_0\}}^{\ve{\alpha}\cup \{\phi_* \alpha_s\}}\circ \phi_*\simeq \begin{pmatrix}\Phi_{\phi_*\ve{\beta}\to\ve{\beta}}^{\ve{\alpha}}\circ \phi_*&*\\
0& \Phi_{\phi_*\ve{\beta}\to\ve{\beta}}^{\ve{\alpha}}\circ \phi_*
\end{pmatrix}.\] Combining this with Equation \eqref{eq:factortransitionmaps}, we see that for sufficiently large $T$, we have
\begin{align*}&\phantom{\simeq} \hspace{3mm}\Phi_{(\bar{\Sigma}_p, \ve{\alpha}\cup \{\alpha_s\},\ve{\beta}\cup \{\beta_0\},\cJ(T))\to (\bar{\Sigma}_{p'}, \ve{\alpha}\cup \{\phi_*\alpha_s\},\ve{\beta}\cup \{\phi_* \beta_0\},\phi_*\cJ(T))}\\&\simeq \begin{pmatrix}\Phi_{(\Sigma, \ve{\alpha},\ve{\beta},J_s)\to (\Sigma, \ve{\alpha},\ve{\beta},\phi_*J_s)}&*\\
0& \Phi_{(\Sigma, \ve{\alpha},\ve{\beta},J_s)\to (\Sigma, \ve{\alpha},\ve{\beta},\phi_*J_s)}
\end{pmatrix}.\end{align*}Since the above map is upper triangular with diagonal entries equal to the change of diagrams maps, the change of diagrams maps commute with the maps  $S_{w,z,\cH,p}^{\pm}$ and $S_{w,z,\cH,p'}^{\pm}$ in the appropriate sense. Since any two choices of points $p$ and $p'$ at which we can perform quasi-stabilization are in the region $A$, we know that the maps on the filtered chain homotopy invariant $CFL_{UV}^\infty$ induced by $S_{w,z,\cH,p,\alpha_s\cJ}^{\pm}$ and $S_{w,z,\cH,p',\phi_*\alpha_s,\phi_* \cJ}^{\pm}$ are equal. Since we proved invariance from the gluing data $\cJ$ in Lemma \ref{lem:changeofalmostcomplexstructure}, and we proved invariance from the curve $\alpha_s$ in Lemma \ref{lem:quasistabindependentofalphas}, the proof is thus complete.\end{proof}

We let $S_{w,z,\cH}^{\pm}$ denote the map $S_{w,z,\cH,p}^{\pm}$ for any choice of $p$  in the component of $\Sigma\setminus \ve{\alpha}$ containing the basepoints of $L$ adjacent to $w$ and $z$.

\begin{lem}If $\cH$ and $\cH'$ are two diagrams for $\bL=(L,\ve{w},\ve{z})$, then the maps $S_{w,z,\cH}^{\pm}$ and $S_{w,z,\cH'}^{\pm}$ are filtered chain homotopic (recall that the maps $S_{w,z,\cH}^{\pm}$ and $S_{w,z,\cH}^{\pm}$ are extended to the entire chain homotopy type invariant by pre- and post-composing with change of diagrams maps).
\end{lem}
\begin{proof}Suppose that $\cH=(\Sigma, \ve{\alpha},\ve{\beta},\ve{w},\ve{z})$ and $\cH'=(\Sigma',\ve{\alpha}',\ve{\beta}',\ve{w},\ve{z})$ are two diagrams for $(L,\ve{w},\ve{z})$. The diagrams $\cH$ and $\cH'$ are related by a sequence of the following moves:
\begin{enumerate}
\item  $\ve{\alpha}$- and $\ve{\beta}$-handleslides and isotopies;
\item (1,2)-stabilizations away from $L$;
\item isotopies of $\Sigma$ inside of $Y$ which fix $\ve{w}\cup \ve{z}$, and map $L$ to $L$.
 \end{enumerate}

 The maps corresponding to $\ve{\alpha}$- and $\ve{\beta}$-handleslides on the unstabilized diagram $\cH$ can be computed using triangle maps. For moves of the $\ve{\beta}$-curves, we simply apply Lemma \ref{lem:singlealphaquasistabtriangle} to see that the maps $S_{w,z,\cH}^{\pm}$ are invariant under $\ve{\beta}$-isotopies and handleslides. Theorem \ref{thm:doublequasitriangle} implies independence under $\ve{\alpha}$-moves of $\cH$ such that there are curves $\alpha_s$ and $\alpha_s'$ in $\Sigma$, with top graded intersection point $p\in \alpha_s'\cap \alpha_s$, such that  $(\Sigma, \ve{\alpha}'\cup \{\alpha_s'\},\ve{\alpha}\cup \{\alpha_s\},\ve{w})$ is strongly positive with respect to $p$. An arbitrary $\ve{\alpha}$-move can be realized as a sequence of such moves, and hence the maps $S_{w,z,\cH}^{\pm}$ are unchanged by handleslides and isotopies of the $\ve{\alpha}$- and $\ve{\beta}$-curves.

The maps $S_{w,z,\cH}^{\pm}$ obviously commute with the $(1,2)$-stabilization maps.

We now consider isotopies $\phi_t:Y\to Y$ which fix $\ve{w}\cup \ve{z}$ and map $L$ to $L$. We note that tautologically we have that
\[\phi_*\circ S_{w,z,\cH,p,\cJ}^{\pm}=S_{w,z,\phi_* \cH, \phi_* p, \phi_* \cJ}^{\pm}\circ \phi_*.\] Since we already know that $S_{w,z,\cH,p,\cJ}^{\pm}$ is invariant from $p$ and $\cJ$, we thus conclude that the maps $S_{w,z,\cH}^{\pm}$ and $S_{w,z,\phi_* \cH}^{\pm}$ agree.
\end{proof}

We can now write $S^{\pm}_{w,z}$ for the quasi-stabilization maps, completing the proof of Theorem \ref{thm:A}.

\begin{rem}Given that the triangle map computations in Lemma \ref{lem:singlealphaquasistabtriangle} and Theorem \ref{thm:doublequasitriangle} showed that change of diagrams maps were not only upper triangular, but diagonal, one may ask why it is natural to define $S_{w,z}^+$ by $\ve{x}\mapsto \ve{x}\times \theta^+$ and not $\ve{x}\mapsto \ve{x}\times \theta^-$. We remark that $\ve{x}\mapsto \ve{x}\times \theta^-$ is only a chain map when $w$ is given the same color as the other $\ve{w}$-basepoint adjacent to $z$, and indeed $\ve{x}\mapsto \ve{x}\times \theta^-$ is equal to $\Psi_zS_{w,z}^+$. The map $\Psi_z$ is only a chain map if the $\ve{w}$-basepoints adjacent to $z$ have the same color.
\end{rem}

\begin{rem}We have defined quasi-stabilization maps $S_{w,z}^{\pm}$ in the case that $w$ comes after $z$ and showed that such maps were invariants, i.e. that they yielded well defined maps $S_{w,z}^{\pm}$ on the coherent filtered chain homotopy type invariant. These maps were only constructed if $w$ came after $z$ on the link component. In the case that $z$ comes after $w$, we can define quasi-stabilization maps $S_{z,w}^\pm$ analogously, picking a choice of $\beta_s$. We could call such an operation $\beta$ quasi-stabilization. There is no ambiguity between $\alpha$ quasi-stabilizations or $\beta$ quasi-stabilizations because $S_{w,z}^{\pm}$ is always an $\alpha$ quasi-stabilization and $S_{z,w}^{\pm}$ is always a $\beta$ quasi-stabilization.
\end{rem}
\section{Commutatation of quasi-stabilization maps}
\label{sec:freestabcommute}
In this section we show that if $\{w,z\}\cap \{w',z'\}=\varnothing$, then the maps $S_{w,z}^{\pm}$ and $S_{w',z'}^{\pm}$ all commute. In \cite{HFplusTQFT} we showed that the free stabilization maps commute, though commutation was easier to show in that setting, since we could just pick a diagram and an almost complex structure where both free-stabilization maps could be computed, and by simply looking at the formulas, one could observe that the maps commuted. In the case of quasi-stabilization, we cannot always pick an almost complex structure which can be used to compute both maps. Nevertheless, we can compute enough components of the change of almost complex structure map to show that quasi-stabilization maps commute:

\begin{thm}\label{thm:quasistabilizationcommutes}Suppose that $(L,\ve{w},\ve{z})$ is a multibased link in $Y^3$ and that $w,z,w'$ and $z'$ are new basepoints, such that $(w,z)$ and $(w',z')$ are each pairs of adjacent basepoints on $(L,\ve{w}\cup \{w,w'\},\ve{z}\cup \{z,z'\})$. Then
\[S_{w,z}^{\circ_1} \circ S_{w',z'}^{\circ_2}\simeq  S_{w',z'}^{\circ_2}\circ S_{w,z}^{\circ_1}\] for any $\circ_1,\circ_2\in \{+,-\}$.
\end{thm}

Pick a diagram $(\Sigma, \ve{\alpha},\ve{\beta},\ve{w},\ve{z})$ for $(L,\ve{w},\ve{z})$, and let $\alpha_s$ and $\alpha_s'$ be curves in $\Sigma\setminus \ve{\alpha}$ along which we can perform quasi-stabilization for $(w,z)$ and $(w',z')$ respectively. Let $\beta_0$ and $\beta_0'$ denote the new $\ve{\beta}$-curves. Let $\cJ$ denote gluing data for stretching along circles bounding $\beta_0$ and $\beta_0'$. There are two distinct cases to consider, corresponding to whether the pairs $(w,z)$ and $(w',z')$ are adjacent or not: either $\alpha_s$ and $\alpha_s'$ lie in the same component of $\Sigma\setminus \ve{\alpha}$ (this case corresponds to having the pair $(w,z)$ be adjacent to the pair $(w',z')$), or $\alpha_s$ and $\alpha_s'$ lie in different components of $\Sigma\setminus \ve{\alpha}$ (this case corresponds to the pair $(w,z)$ not being adjacent to $(w',z')$).

The first case is the easier to consider. In this case, we now show that we can pick an almost complex structure which computes both quasi-stabilization maps. To this end, we have the following lemma:

\begin{lem}\label{lem:quasistabilizationcommutesI} Suppose that $(\Sigma, \ve{\alpha},\ve{\beta},\ve{w},\ve{z})$ is a diagram as in the previous paragraph with new curves $\alpha_s$ and $\alpha_s'$ for quasi-stabilizing at $(w,z)$ and $(w',z')$ respectively. If $\alpha_s$ and $\alpha_s'$ are not in the same component of $\Sigma\setminus \ve{\alpha}$, then for all sufficiently large $T_1,T_1',T_2,T_2'$, we have
\[\Phi_{\cJ(T_1,T_1')\to \cJ(T_2,T_2')}\simeq \id,\] with respect to the obvious identification between the two complexes.
\end{lem}

\begin{proof}
To show this, we will perform a computation similar to the one in Proposition \ref{lem:differentialcomp}, but for Maslov index zero disks. Let $A$ be the component of $\Sigma\setminus \ve{\alpha}$ which contains $\alpha_s$ and let $A'$ denote the component of $\Sigma\setminus \ve{\alpha}$ which contains $\alpha_s'$. By assumption $A\neq A'$. Let $A_1$ and $A_2$ denote the two components of $A\setminus \alpha_s$. Let $A_1'$ and $A_2'$ denote the components of $A'\setminus \alpha_s'$. Suppose that $\phi$ is a Maslov index 0 homology disk on the doubly quasi-stabilized diagram $(\bar{\Sigma},\ve{\alpha}\cup \{\alpha_s,\alpha_s'\},\ve{\beta}\cup \{\beta_0,\beta_0'\},\ve{w}\cup \{w,w'\},\ve{z}\cup \{z,z'\})$. Write $\phi=\phi_\Sigma\# \phi_0\# \phi_0'$ where $\phi_\Sigma$ is a homology class on $(\Sigma, \ve{\alpha}\cup \{\alpha_s,\alpha_s'\},\ve{\beta},\ve{w},\ve{z})$, $\phi_0$ is a homology class on $(S^2,\alpha_0,\beta_0,w,z)$ and $(S^2,\alpha_0',\beta_0',w,z)$. Suppose that $T_{1,n},T_{1,n}',T_{2,n},$ and $T_{2,n}'$ are sequences of neck lengths all approaching $+\infty$ and let $\hat{\cJ}_n$ denote an interpolating almost complex structures between $\cJ(T_{1,n},T_{1,n}')$ and $\cJ(T_{2,n},T_{2,n}')$. Pick $\hat{\cJ}_n$ so that as $n\to \infty$ the almost complex structures $\hat{\cJ}_n$ split into $J_s\vee J_{S^2}\vee J_{S^2}$ on $(\Sigma \vee S^2\vee S^2)\times [0,1]\times \R$. If $u_n$ is a sequence of Maslov index 0, $\hat{\cJ}_n$-holomorphic curves representing $\phi$, we can extract a weak limit to broken curves $U_\Sigma$, $U_0$ and $U_0'$  on $(\Sigma, \ve{\alpha}\cup \{\alpha_s,\alpha'_s\},\ve{\beta},\ve{w},\ve{z})$, $(S^2,\alpha_0,\beta_0)$ and $(S^2,\alpha_0',\beta_0')$ representing $\phi_\Sigma, \phi_0$ and $\phi_0'$ respectively. As in Proposition \ref{lem:differentialcomp}, the curves in $U_\Sigma$ consist of a broken holomorphic strip $U_\Sigma'$ on $(\Sigma, \ve{\alpha},\ve{\beta})$, as well as a collection $\cA$ of cylindrical $(\ve{\alpha}\cup \{\alpha_s\}\cup \{\alpha_s'\})$-boundary degenerations. Let $\phi_\Sigma'$ denote the underlying homology class of $U_\Sigma'$. Let $m_1,m_2,n_1,n_2,m_1',m_2',n_1',$ and $n_2'$ be multiplicities as in Figure \ref{fig::14}. 

\begin{figure}[ht!]
\centering
\includegraphics[scale=1.2]{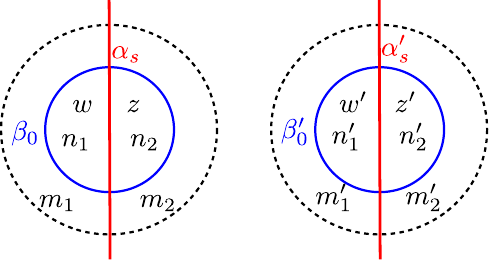}
\caption{Multiplicities of a disk $\phi$ near new basepoints $w,z,w'$ and $z'$ on a diagram which has been quasi-stabilized twice. \label{fig::14}}
\end{figure}

Adapting the Maslov index computation from Proposition \ref{lem:differentialcomp}, we see that
\[\mu(\phi)=\mu(\phi'_\Sigma)+n_1(\phi)+n_2(\phi)+n_1'(\phi)+n_2'(\phi)\]\[+m_1(\cA)+m_2(\cA)+m_1'(\cA)+m_2'(\cA)+2\sum_{\cD\in C(\Sigma\setminus \ve{\alpha}), \cD\neq A} n_{\cD}(\cA)\] where $C(\Sigma\setminus \ve{\alpha})$ denotes the connected components of $\Sigma\setminus \ve{\alpha}$.

Since $U_\Sigma'$ is a broken holomorphic curve for an $\R$-invariant almost complex structure, we conclude that $U_\Sigma'$ consists only of constant flowlines. Since all of the other summands are zero, it's easy to see that this forces all multiplicities to be zero. Hence only constant disks are counted by the change of almost complex structures map, concluding the proof of the lemma.
\end{proof}

In the case that $\alpha_s$ and $\alpha_s'$ are in the same component, the change of almost complex structure maps will often be nontrivial. Nevertheless we have the following:

\begin{lem}\label{lem:quasistabilizationcommutesII}Suppose that the pairs $(w',z')$ and $(w,z)$ are adjacent on $(L,\ve{w}\cup \{w,w'\},\ve{z}\cup \{z,z'\})$ and that $(w',z')$ comes after $(w,z)$. Let $\theta^{\pm}$ and $(\theta')^{\pm}$ denote the intersection points corresponding to quasi-stabilization. If  $T_1,T_1',T_2,T_2'$ are all sufficiently large, then writing $F=\Phi_{\cJ(T_1,T_1')\to \cJ(T_2,T_2')}$, we have

\begin{align*}F(\ve{x}\times \theta^+\times (\theta')^+)&=\ve{x}\times \theta^+\times (\theta')^+\\ F(\ve{x}\times \theta^+\times (\theta')^-)&=\ve{x}\times \theta^+\times (\theta')^-+C\cdot \ve{x}\times \theta^-\times(\theta')^+\\
F(\ve{x}\times \theta^-\times (\theta')^+)&=\ve{x}\times \theta^-\times(\theta')^+ \\
F(\ve{x}\times\theta^-\times(\theta')^-)&=\ve{x}\times\theta^-\times(\theta')^-,
\end{align*}
for some $C$ (which is not independent of $T_i$ and $T_i'$). 
\end{lem}

\begin{proof}We proceed similar to the previous lemma. Now a single component, which we denote by $A$, of $\Sigma\setminus \ve{\alpha}$ contains both $\alpha_s$ and $\alpha_s'$. Write $A_1,A_2$ and $A_3$ for the three different components of $A\setminus (\alpha_s\cup \alpha_s')$. Two of the $A_i$ share boundary with exactly one of the other $A_j$, and one of the $A_i$ shares boundary with both of the other $A_i$. Without loss of generality assume that $A_1$ shares boundary with $A_2$, and that $A_2$ also shares boundary with $A_3$.

As before, as we simultaneously stretch the necks, a sequence of Maslov index zero disks $u_i$ has a weak limit as before. Now, however, the Maslov index computation is different. Let $a_i(\cA)$ denote the multiplicity of the $(\ve{\alpha}\cup \{\alpha_s\}\cup \{\alpha_s'\})$-degeneration $\cA$ in the region $A_i$. One computes that the Maslov index now satisfies

\[\mu(\phi)=\mu(\phi'_\Sigma)+n_1(\phi)+n_2(\phi)+n_1'(\phi)+n_2'(\phi)+a_1(\cA)+a_3(\cA)+2\sum_{\cD\in C(\Sigma\setminus \ve{\alpha}), \cD\neq A} n_{\cD}(\cA).\] As usual this implies that all of the above terms are zero. Hence $\phi_\Sigma'$, which has a broken representative for a cylindrical almost complex structure, must be the constant disk by transversality. The only multiplicities which may be nonzero are $a_2(\cA)$, and the multiplicities $m_i(\phi)$ and $m_i'(\phi)$, none of which appear in the above sum. As is easily observed, this constrains the disk $\phi$ to be in $\pi_2(\ve{x}\times \theta^+\times (\theta')^-,\ve{x}\times \theta^-\times (\theta')^+)$, completing the proof. An example of a disk which might appear in the change of almost complex structure map is shown in Figure \ref{fig::13}.
\end{proof}

\begin{figure}[ht!]
\centering
\includegraphics[scale=1.2]{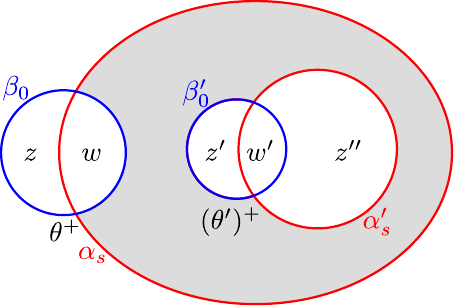}
\caption{An example of a Maslov index zero disk which might be counted by $\Phi_{\cJ(T_1,T_1')\to \cJ(T_2,T_2')}$ in Lemma \ref{lem:quasistabilizationcommutesII} for arbitrarily large $T_i$ and $T_i'$.} \label{fig::13}
\end{figure}

\begin{proof}[Proof of Theorem \ref{thm:quasistabilizationcommutes}] The proof is easy algebra in all cases using Lemmas \ref{lem:quasistabilizationcommutesI} and \ref{lem:quasistabilizationcommutesII}.
\end{proof} 

\section{Further relations between the maps \texorpdfstring{$\Psi_z,\Phi_w$}{Psi\_z, Phi\_w} and \texorpdfstring{$S^{\pm}_{w,z}$}{S\^{}pm\_ w,z}}
\label{sec:furtherrelations}
In this section we prove several relations between the maps $S_{w,z}^{\pm}$, $\Psi_z$ and $\Phi_w$. We highlight the convenience of viewing $\Phi_w$ and $\Psi_z$ as formal derivatives of the differential, since basically all of the relations in this section are derived by either formally differentiating the expression for $\d\circ \d$ from Lemma \ref{lem:del^2=}, or by differentiating our expression of the quasi-stabilized differential in Proposition \ref{lem:differentialcomp}.

\begin{lem}\label{lem:commutatorsofpsiandphi}If $w$ and $z$ are not adjacent, or if $w$ and $z$ are the only basepoints on a link component, then \[\Phi_w\Psi_z+\Psi_z\Phi_w\simeq 0.\] If $w$ and $z$ are adjacent and there are other basepoints on the link component, then \[\Phi_w\Psi_z+\Psi_z\Phi_w\simeq 1.\]
\end{lem}
\begin{proof}Take the expression for $\d^2$ from Lemma \ref{lem:del^2=} and differentiate it with respect to $U_w$. We obtain
\[\d \Phi_w+\Phi_w\d=V_{z'}+V_{z''}\] where $z'$ and $z''$ are the variables adjacent to $w$ on its link component. Suppose first that $z'\neq z''$, i.e. that $w$ and $z$ are not the only basepoints on their link component.

Differentiating the above expression with respect to $V_z$, we see that
\[\Psi_z\Phi_w+\Phi_w\Psi_z\simeq \begin{cases}1& \text{ if } w \text{ is adjacent  to } z\\
0& \text{ if } w \text{ is not adjacent to } z.
\end{cases},\] from which the claim follows as long as $w$ and $z$ are not the only basepoints on their link component.

If $w$ and $z$ are the only basepoints on their link component, then $z=z'=z''$ and the above argument is easily modified to give the stated result.
\end{proof}

\begin{lem}\label{lem:somemorecommutators} We have
\[\Psi_z\Psi_{z'}+\Psi_{z'}\Psi_z\simeq 0\] and
\[\Phi_w\Phi_{w'}+\Phi_{w'}\Phi_w\simeq 0\] for any choice of $z,z',w,$ and $w'$.
\end{lem}
\begin{proof}This is proven identically to the previous lemma.
\end{proof}

As with the free stabilization maps in \cite{HFplusTQFT}, we have the following:

\begin{lem}\label{lem:del'appears} The following relation holds: \[S_{w,z}^+ S_{w,z}^-=\Phi_w.\]
\end{lem}

\begin{proof}The differential on the uncolored quasi-stabilized diagram takes the form
\[\d_{\bar{\cH}}=\begin{pmatrix}\d_\cH & U_w+U_{w'}\\
V_z+V_{z'}& \d_\cH
\end{pmatrix}\] where $\d_\cH$ is the differential on the unstabilized diagram. After taking the $U_w$ derivative we get 
\[\Phi_w=\begin{pmatrix}0& 1\\
0& 0\end{pmatrix},\] which is exactly $S_{w,z}^+ S_{w,z}^-$.
\end{proof}

We now consider commutators of the quasi-stabilization maps and the maps $\Phi_w$ and $\Psi_z$.

\begin{lem}\label{Spm-Psicommutatornotadjacent}Suppose that $w$ and $z$ are two new basepoints for a link. If $w$ is not adjacent to $z'$, then
\[S^{\pm}_{w,z} \Psi_{z'}+\Psi_{z'} S^{\pm}_{w,z}\simeq 0.\] With no assumptions on adjacency, we have
\[S^{\pm}_{w,z}\Phi_{w'}+\Phi_{w'} S^{\pm}_{w,z}\simeq 0.\]
\end{lem}
\begin{proof}Suppose that $w$ and $z$ are inserted between basepoints $z''$ and $w''$ on the link $\bL$. The quasi-stabilized differential takes the form

\[\d_{\bar{\cH}}=\begin{pmatrix}\d_{\cH} & U_w+U_{w''}\\
V_z+V_{z''}& \d_{\cH}
\end{pmatrix},\] by Proposition \ref{lem:differentialcomp}. By assumption $z'\neq z''$. Differentiating with respect to $V_{z'}$ thus yields
\[\tilde{\Psi}_{z'}=\begin{pmatrix} \Psi_{z'}& 0\\
0& \Psi_{z'}
\end{pmatrix},\] where $\tilde{\Psi}_{z'}$ denotes the map on the stabilized diagram, and $\Psi_{z'}$ denotes the map on the unstabilized diagram. In matrix notation the maps $S^{\pm}_{w,z}$ take the form 
\begin{equation}S_{w,z}^+=\begin{pmatrix}1\\0
\end{pmatrix}\qquad \text{ and } \qquad S_{w,z}^-=\begin{pmatrix}0& 1\end{pmatrix}.\label{eq:S^pmmatrixnotation}
\end{equation} The stated equality involving $\Psi_{z'}$ now follows from matrix multiplication.

The equality involving $\Phi_{w'}$ follows similarly.
\end{proof}

We also have the following:

\begin{lem}\label{SpmPsicommutatoradjacent}Suppose that $z'$ is adjacent to $w$ and that $z'\neq z$. Then we have
\[S_{w,z}^+ \Psi_{z'}\simeq (\Psi_{z'}+\Psi_z)S^+_{w,z}\] and
\[\Psi_{z'}S^-_{w,z}\simeq S_{w,z}^-(\Psi_{z'}+\Psi_z).\]
\end{lem}
\begin{proof}Once again we consider the quasi-stabilized differential, which is
\[\d_{\bar{\cH}}=\begin{pmatrix}\d_\cH & U_w+U_{w'}\\
V_z+V_{z'}& \d_\cH
\end{pmatrix}.\] Differentiating with respect to $z'$ yields
\[\tilde{\Psi}_{z'}=\begin{pmatrix}\Psi_{z'}&0\\
1& \Psi_{z'}\end{pmatrix}\] and
\[\tilde{\Psi}_z=\begin{pmatrix}0&0\\
1& 0
\end{pmatrix}.\] Here $\tilde{\Psi}_z$ denotes the map on the complex after quasi-stabilization, and $\Psi_z$ denotes the map on the complex before quasi-stabilization. Using the matrix notation from Equation \eqref{eq:S^pmmatrixnotation}, the desired relations follow from matrix multiplication.
\end{proof}

 The reader can compare the following lemma to \cite{HFplusTQFT}*{Lem. 7.7}, the analogous result for the closed three manifold invariants.

\begin{lem}\label{lem:addtrivialstrand} Suppose that $\bL=(L,\ve{w},\ve{z})$ is a multibased link in $Y^3$, such that $w$ and $z$ are two new, consecutive basepoints on $\bL$ such that $w$ follows $z$. If $z'$ is one of the two $\ve{z}$-basepoints adjacent to $w$, then we have \[S^-_{w,z} \Psi_{z'} S^+_{w,z}\simeq 1.\] 
\end{lem}
\begin{proof}This follows from our usual strategy. Pick a diagram $\cH$ for $(L,\ve{w},\ve{z})$ and let $\bar{\cH}$ denote a diagram which has been quasistabilized at $w$ and $z$. Let $z''$ and $w''$ denote the basepoints adjacent to $w$ and $z$ on $\bL$. Using Proposition \ref{lem:differentialcomp} we have that
\[\d_{\bar{\cH}}=\begin{pmatrix}\d_{\cH}&U_{w}+U_{w''}\\
V_{z}+V_{z''}& \d_{\cH}
\end{pmatrix}.\] By assumption, either $z'=z$ or $z'=z''$ (but not both). In both cases, we have that
\[\Psi_{z'}=\left(\frac{d}{d V_{z'}} \d_{\bar{\cH}}\right)=\begin{pmatrix}*& *\\1 &*
\end{pmatrix},\] where the $*$ terms are unimportant. Using the matrix notation from Equation \eqref{eq:S^pmmatrixnotation}, we get the desired equality immediately from matrix multiplication. 
\end{proof}

The reader should compare the following to \cite{Smovingbasepoints}*{Lem. 4.4.}.

\begin{lem}\label{lem:psi^2=0}We have $\Psi_z^2\simeq 0$ and $\Phi_w^2\simeq 0$, as $\frP$-filtered maps of $\Z_2[U_\frP]$-modules.
\end{lem}

\begin{proof}The proof follows identically to the proof of \cite{HFplusTQFT}*{Prop. 14.13.}.
\end{proof}

\section{Basepoint moving maps}

\label{sec:basepointmovingmaps}
In this section, we compute several basepoint moving maps. The procedure for computing maps induced by moving basepoints is in a similar spirit to the author's computation of the $\pi_1$-action on the Heegaard Floer homology of a closed three manifold in \cite{HFplusTQFT}. We first compute the effect of moving basepoints along a small arc on a link component via a model computation. We then use this to prove Theorem \ref{thm:B}, the effect of moving all of the basepoints on a link component in one full loop. The final computation is Theorem \ref{thm:D}, where we compute the effect on certain colored complexes of moving each $\ve{w}$-basepoint to the next $\ve{w}$-basepoint, and moving each $\ve{z}$-basepoint to the next $\ve{z}$-basepoint.

\subsection{Moving basepoints along an arc}

Suppose we $Y^3$ is a three manifold with embedded multipointed link $\bL_0=(L,\ve{w}_0,\ve{z}_0)$, though we allow the case that one of the components of $\bL_0$ has no basepoints. Suppose  that $z,w,z',w'$ are all points on a single component of $L\setminus (\ve{w}_0\cup \ve{z}_0)$, appearing in that order according to the orientation of $\bL$. Let 
\[\ve{w}=\ve{w}_0\cup \{w\}, \qquad \ve{z}=\ve{z}_0\cup \{z\}, \qquad \ve{w}'=\ve{w}_0\cup \{w'\}, \qquad \text{ and } \qquad \ve{z}'=\ve{z}_0\cup \{z'\}.\] Finally assume that $(L,\ve{w},\ve{z})$ has basepoints in each component of $L$. There is an isotopically unique diffeomorphism of $Y$ which maps $L$ to itself and fixes $\ve{w}_0\cup \ve{z}_0$ and maps $w$ to $w'$ and $z$ to $z'$, which is isotopic to the identity relative $\ve{w}_0\cup \ve{z}_0$ through isotopies which map $L$ to itself. Let $\varsigma_0$ denote this diffeomorphism. The diffeomorphism $\varsigma_0$ induces a map
\[(\varsigma_0)_*:CFL^\infty_{UV}(Y,L,\ve{w},\ve{z},\frs)\to CFL^\infty_{UV}(Y,L,\ve{w}',\ve{z}',\frs).\]

\begin{lem}\label{lem:movebasepoints}The induced map $(\varsigma_0)_*$ is filtered chain homotopic to
\[(\varsigma_0)_*\simeq S_{w,z}^- \Psi_{z'} S_{w',z'}^+.\]
\end{lem}
\begin{proof}We first prove the result in the case that the link component containing $w$ and $z$ has at least one extra pair of basepoints. Let $z''$ denote the basepoint occurring immediately after $w'$. In this case, we can pick a diagram like the one shown in Figure \ref{fig::10}, where the dashed lines show two circles along which the almost complex structure will be stretched. In this diagram, we assume that $\alpha_s$ and $\alpha_s'$ each bound disks on $\Sigma$ and that $\alpha_s,\alpha_s',\beta_0$ and $\beta_0'$ do not intersect any other $\ve{\alpha}$- or $\ve{\beta}$-curves. With this diagram, we can compute all of the maps $\Psi_{z'},S_{w,z}^-,$ and $S_{w',z'}^+$ explicitly. We must be careful though, since we cannot use the same almost complex structure for all of the maps. Instead we will need to use the change of almost complex structure computation from Lemma \ref{lem:quasistabilizationcommutesII}. Let $J_s$ be an almost complex structure which is sufficiently stretched along $c$ to compute compute $S_{w,z}^{\pm}$, and let $J_s'$ be an almost complex structure which is sufficiently stretched along $c'$ to compute $S_{w',z'}^{\pm}$, and assume that both are stretched sufficiently so that the change of almost complex structure map $\Phi_{J_s'\to J_s}$ takes the form described in Lemma \ref{lem:quasistabilizationcommutesII}. We wish to compute $S^-_{w,z}\circ \Psi_{z'} \circ \Phi_{J_s'\to J_s}\circ S^+_{w',z'}$.

 Let $\theta^{\pm}$ denote the intersection points of $\alpha_s\cap \beta_0$ and let $(\theta')^{\pm}$ denote the intersection points of $\alpha_s'$ and $\beta_0'$.

\begin{figure}[ht!]
\centering
\includegraphics[scale=1.2]{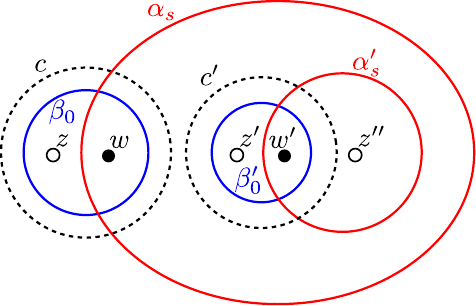}
\caption{A diagram for Lemma \ref{lem:movebasepoints} when we have another basepoint $z''$ on the link component containing $w,z,w'$ and $z'$. The curves $\alpha_s$ and $\alpha'_s$ each bound disks, and $\alpha_s,\alpha_s',\beta_0,\beta_0'$ do not intersect any other $\ve{\alpha}$- or $\ve{\beta}$-curves.} \label{fig::10}
\end{figure}

Using the analysis in Proposition \ref{lem:differentialcomp}, we see that for $J_s$ there are exactly two domains which are the domain of Maslov index one disks $\phi$ which support holomorphic representatives  with $n_{z'}(\phi)>0$. These  domains are shown in Figure \ref{fig::11}. Also every homology disk $\phi$ which has one of these domains has $\#\hat{\cM}_{J_s}(\phi)=1$.

\begin{figure}[ht!]
\centering
\includegraphics[scale=1.2]{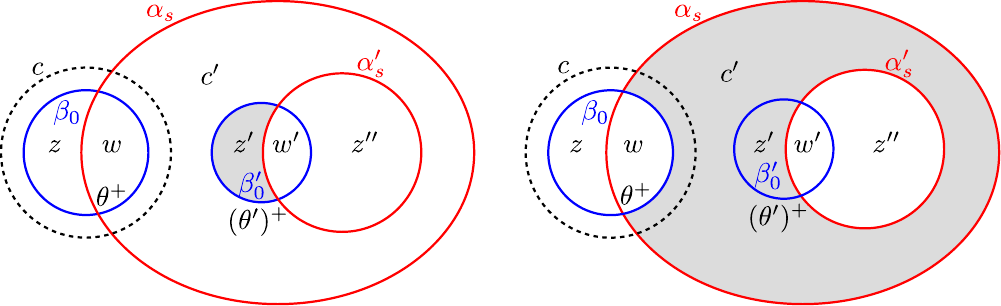}
\caption{The two domains contributing to $\Psi_{z'}$ in Lemma \ref{lem:movebasepoints} for the almost complex structure $J_s$ stretched sufficiently along $c$.} \label{fig::11}
\end{figure}

We wish to show that $S_{w,z}^-\circ \Psi_{z'}\circ \Phi_{J'_s\to J_s} \circ S^+_{w',z'}=(\varsigma_0)_*$, where $\varsigma_0$ is the diffeomorphism induced by simply pushing $w$ and $z$ to $w'$ and $z'$ respectively. To this end, it is sufficient to show that

\[(S_{w,z}^-\circ \Psi_{z'}\circ \Phi_{J'_s\to J_s} \circ S^+_{w',z'})(\ve{x}\times \theta^{\pm})=\ve{x}\times (\theta')^{\pm},\] since $(\varsigma_0)_*(\ve{x}\times \theta^{\pm})=\ve{x}\times (\theta')^{\pm}$.

Homology disks with the left domain in Figure \ref{fig::11} yield a contribution of
\[\ve{x}\times \theta^{\pm}\times (\theta')^{+}\longrightarrow \ve{x}\times \theta^{\pm}\times (\theta')^{-}\] to $\Psi_{z'}$.

Homology disks with the right domain in Figure \ref{fig::11} yield a contribution of
\[\ve{x}\times \theta^{+}\times (\theta')^{\pm}\longrightarrow \ve{x}\times \theta^{-}\times (\theta')^{\pm}\] to $\Psi_{z'}$.

We first compute $(S_{w,z}^-\circ \Psi_{z'}\circ \Phi_{J'_s\to J_s} \circ S^+_{w',z'})(\ve{x}\times\theta^+)$. Using the above computation of $\Psi_{z'}$ and the computation of $\Phi_{J'_s\to J_s}$ from Lemma \ref{lem:quasistabilizationcommutesII}, we have that
\begin{align*}(S_{w,z}^-\circ \Psi_{z'}\circ \Phi_{J'_s\to J_s} \circ S^+_{w',z'})(\ve{x}\times\theta^+)&=(S_{w,z}^-\circ \Psi_{z'}\circ \Phi_{J'_s\to J_s})(\ve{x}\times \theta^+\times (\theta')^+)\\
&=(S_{w,z}^-\circ \Psi_{z'})(\ve{x}\times \theta^+\times (\theta')^+)\\
&= S_{w,z}^-(\ve{x}\times \theta^+\times (\theta')^-+\ve{x}\times \theta^-\times (\theta')^+)\\
&=\ve{x}\times (\theta')^+\end{align*}

We now compute $(S_{w,z}^-\circ \Psi_{z'} \circ \Phi_{J'_s\to J_s}\circ S^+_{w',z'})(\ve{x}\times\theta^-)$. Once again using our previous computation of $\Psi_{z'}$ and Lemma \ref{lem:quasistabilizationcommutesII}, we have that
\begin{align*}(S_{w,z}^-\circ \Psi_{z'} \circ \Phi_{J'_s\to J_s} \circ S^+_{w',z'})(\ve{x}\times\theta^-)&=(S_{w,z}^- \circ \Psi_{z'} \circ\Phi_{J'_s\to J_s})(\ve{x}\times \theta^-\times (\theta')^+)\\
&=(S_{w,z}^-\circ \Psi_{z'})(\ve{x}\times \theta^-\times (\theta')^+)\\
&=S_{w,z}^-(\ve{x}\times \theta^-\times (\theta')^-)\\
&=\ve{x}\times (\theta')^-,\end{align*}
completing the proof of the claim if $\ve{w}$ and $\ve{z}$ each have at least two basepoints on $L$.

We now consider the case that $\bL$ doesn't have any basepoints other than $w$ and $z$. In this case we just introduce two new basepoints $w'',z''$ which are on the component of $L\setminus \{w,w',z,z'\}$ which is going from $w'$ to $z$.  Note that $\varsigma_0$ is isotopic relative $\{w,z\}$ to a diffeomorphism which fixes $w''$ and $z''$. Hence $(\varsigma_0)_*S_{w'',z''}^-=S_{w'',z''}^-(\varsigma_0)_*$. We just compute that
\begin{align*}(\varsigma_0)_*& \simeq(\varsigma_0)_* (S_{w'',z''}^- \Psi_{z''} S_{w'',z''}^+)&&\text{(Lemma \ref{lem:addtrivialstrand})}\\
&\simeq S_{w'',z''}^-(\varsigma_0)_*\Psi_{z''} S_{w'',z''}^+&&\text{(observation above)}\\
&\simeq S_{w'',z''}^-(S_{w,z}^- \Psi_{z'} S_{w',z'}^+)\Psi_{z''} S_{w'',z''}^+&&\text{(previous case)}\\
&\simeq(S_{w,z}^- \Psi_{z'} S_{w',z'}^+)(S_{w'',z''}^-\Psi_{z''} S_{w'',z''}^+)&&(\text{Theorem } \ref{thm:quasistabilizationcommutes}, \text{ Lemma } \ref{Spm-Psicommutatornotadjacent})\\
&\simeq S_{w,z}^- \Psi_{z'} S_{w',z'}^+,&&\text{(Lemma \ref{lem:addtrivialstrand})}
\end{align*}

as we wanted.

\end{proof}

\subsection{Sarkar's formula for moving basepoints in a full twist around a link component}
\label{sec:movingmanybasepoints}

In this section, we prove Theorem \ref{thm:B}, which is Sarkar's conjectured formula for the effect of moving basepoints on a link component in a full twist around the link component for the full link Floer complex. The main technical tool is Lemma \ref{lem:movebasepoints}, which computes the effect of moving basepoints on a small arc on a link component. By writing the diffeomorphism of a full twist as a composition of many smaller moves of the previous form, we will obtain Sarkar's formula.

\begin{customthm}{\ref{thm:B}}Suppose $\varsigma$ is the diffeomorphism corresponding to a positive Dehn twist around a link component $K$ of $\bL$. Suppose that the basepoints on $K$ are $w_1,z_1,\dots, w_n, z_n$. The induced map $\varsigma_*$ on $CFL_{UV}^\infty(Y,\bL,\sigma, \frP,\frs)$ has the $\frP$-filtered $\Z_2[U_\frP]$ chain homotopy type
\[\varsigma_*\simeq 1+\Phi_K\Psi_K\] where 
\[\Phi_K= \sum_{j=1}^n \Phi_{w_j}\qquad \text{ and } \qquad \Psi_K=\sum_{j=1}^n \Psi_{z_j}.\]
\end{customthm}

To the reader who is not interested in colorings, we note that one can just take $\frP=\ve{w}\cup|L|$, where $|L|$ denotes the components of $L$.

In this section, we also introduce some new formalism to make the computation easier. The maps $\Psi_{z'}$ and $S_{w,z}^\pm$ interact strangely (e.g. Lemma \ref{SpmPsicommutatoradjacent}), which leads to challenging and messy algebra if we are not careful. Suppose that $A$ is an arc on $L$ between two $\ve{w}$-basepoints which share the same color. We define the map
\[\Psi_A=\sum_{z\in A\cap \ve{z}} \Psi_z.\] The maps $\Psi_A$ can be thought of as defining an action of \[\Lambda^* H_1(L/(\ve{w},\sigma);\Z)\] on $CFL_{UV}^\infty(Y,\bL,\sigma, \frP,\frs)$, where $L/(\ve{w},\sigma)$ denotes the space obtained by identifying two $\ve{w}$-basepoints if they share the same color. This formalism is intriguing, but we will only have use for maps $\Psi_A$ for arcs $A$ between $\ve{w}$-basepoints of the same color.

Given an arc $A$ between two $\ve{w}$-basepoints, we define an \textbf{endpoint} of $A$ to be a basepoint $w$ such that the sets $\bar{K\setminus A}$ and $\bar{A}$ both contain $w$ (so if $A=K$, $A$ has no endpoints).

We now proceed to prove some basic properties of the maps $\Psi_A$, all of which are recastings of previous lemmas proven about the maps $\Psi_z$. 

\begin{lem}\label{lem:quasistabcommuteswithpsi} We have \[S_{w,z}^{\pm} \Psi_A+\Psi_A S_{w,z}^{\pm}\simeq 0,\] as long as $w$ is not an endpoint of $A$.
\end{lem}
\begin{proof} This follows immediately from Lemmas \ref{Spm-Psicommutatornotadjacent} and \ref{SpmPsicommutatoradjacent}.
\end{proof}

\begin{lem}\label{lem:Psi_Ascommute}If $A$ and $A'$ are two arcs between $\ve{w}$-basepoints, then
\[\Psi_A\Psi_{A'}+\Psi_{A'}\Psi_A\simeq 0.\]
\end{lem}

\begin{proof} This follows from Lemma \ref{lem:somemorecommutators}.
\end{proof}

\begin{lem}\label{lem:Psi_A^2=0}If $A$ is an arc on $L$, then we have
\[\Psi_A^2\simeq 0,\] as filtered equivariant maps.
\end{lem}
\begin{proof}Simply write $\Psi_{A}=\sum_{z\in A\cap \ve{z}} \Psi_z$, multiply out $\Psi_A^2$, then apply Lemmas \ref{lem:somemorecommutators} and \ref{lem:psi^2=0}.
\end{proof}

\begin{lem}\label{lem:psi_Kpsi_A} Suppose $A\subset K$ is an arc between $\ve{w}$-basepoints and let $c(A)$ denote the arc $\bar{K\setminus A}$. Then
\[\Psi_K \Psi_A=\Psi_{c(A)}\Psi_{A}.\]
\end{lem}
\begin{proof} Write $\Psi_K=\Psi_A+\Psi_{c(A)}$ and then use the previous lemma to compute that
\[\Psi_K\Psi_A=(\Psi_A+\Psi_{c(A)})\Psi_A=\Psi_A^2+\Psi_{c(A)}\Psi_A=\Psi_{c(A)}\Psi_A.\]
\end{proof}

\begin{lem}\label{lem:wendpointcommutator} If $w$ is an endpoint of $A$ then we have
\[\Psi_A \Phi_w+\Phi_w \Psi_A\simeq 1.\] If $w$ is not an endpoint of $A$, then we have
\[\Psi_A\Phi_w+\Phi_w\Psi_A\simeq 0.\]
\end{lem}
\begin{proof} The first claim follows from Lemma \ref{lem:commutatorsofpsiandphi}. The second claim follows from Lemma \ref{lem:quasistabcommuteswithpsi} since we can always write $\Phi_w=S_{w,z}^+S_{w,z}^-$ for $z$ the basepoint immediately preceding $w$ on $\bL$.
\end{proof}

We can now prove Theorem \ref{thm:B}:
\begin{proof}[Proof of Theorem \ref{thm:B}]Let $w_1,z_1,\dots w_n,z_n$ be the basepoints on $K$, in the reverse order that they appear on $K$ according to the orientation of $K$. Let $w_1'z_1',\dots w_n',z_n'$ be new basepoints on $K$ in the interval between $z_n$ and $w_1$. Let $A_j$ be the arc on $K$ from $w_j$ to $w_j'$, as in Figure \ref{fig::20}.

 \begin{figure}[ht!]
\centering
\includegraphics[scale=1.2]{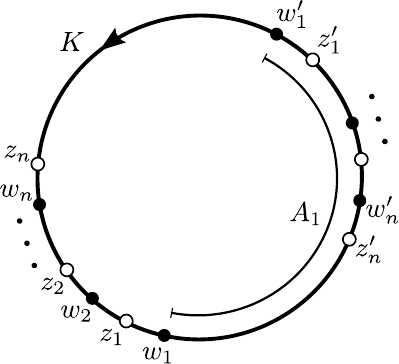}
\caption{The basepoints $z_1,w_1,\dots, z_n,w_n$ and $z_1',w_1',\dots z_n',w_n'$, as well as the arcs $A_i$. \label{fig::20}}
\end{figure}

Write
\[\ve{w}=\{w_1,\dots,w_n\},\qquad \ve{z}=\{z_1,\dots, z_n\},\qquad\ve{w}'=\{w_1',\dots, w_n'\},\quad \text{ and }\quad \ve{z}'=\{z_1',\dots, z_n'\}.\]As usual, we write $\varsigma$ as a composition of two diffeomorphisms $\varsigma=\varsigma_2\circ \varsigma_1$, where $\varsigma_1$ moves the basepoints $\ve{w}$ and $\ve{z}$ to $\ve{w}'$ and $\ve{z}'$ respectively and $\varsigma_2$  moves the basepoints $\ve{w}'$ and $\ve{z}'$ to $\ve{w}$ and $\ve{z}$ respectively. Let $c(A_i)=\bar{K\setminus A_i}$.  Note that by Lemma \ref{lem:addtrivialstrand} we have \begin{equation}
\prod_{j=1}^n (S_{w_j',z_j'}^-\Psi_{A_j} S_{w_j',z_j'}^+)\simeq 1\label{eq:manytrivialstrands}.
\end{equation}

Write $S_{\ve{w},\ve{z}}^{\pm}$ for $\prod_{j=1}^n S_{w_j,z_j}^{\pm}$ and similarly for $S_{\ve{w}',\ve{z}'}^{\pm}$. We compute as follows: 
\begin{align*}\varsigma_*&=(\varsigma_2)_*\circ (\varsigma_1)_*&&\\
&=\bigg(\prod_{j=1}^n S_{w_j',z_j'}^-\Psi_{c(A_j)} S_{w_j,z_j}^+\bigg)\bigg(\prod_{j=1}^n S_{w_j,z_j}^-\Psi_{A_j} S_{w_j',z_j'}^+\bigg)&& \text{(Lemma \ref{lem:movebasepoints})}\\
&=S_{\ve{w}',\ve{z}'}^- \bigg(\prod_{j=1}^n \Psi_{c(A_j)}\bigg)S_{\ve{w},\ve{z}}^+ S_{\ve{w},\ve{z}}^- \bigg(\prod_{j=1}^n \Psi_{A_j}\bigg)S_{\ve{w}',\ve{z}'}^+&& \text{(Lemmas \ref{lem:quasistabcommuteswithpsi},  \ref{thm:quasistabilizationcommutes})}\\
&=S_{\ve{w}',\ve{z}'}^- \bigg(\prod_{j=1}^n \Psi_{c(A_j)}\bigg)\bigg(\prod_{j=1}^n \Phi_{w_j}\bigg) \bigg(\prod_{j=1}^n \Psi_{A_j}\bigg)S_{\ve{w}',\ve{z}'}^+&&\text{(Lemmas \ref{lem:del'appears}, \ref{thm:quasistabilizationcommutes})}\\&=S_{\ve{w'},\ve{z}'}^- \sum_{s\in \{0,1\}^n}\bigg(\prod_{j=1}^n \Phi_{w_j}^{s_j}\bigg)\bigg(\prod_{j=1}^n \Psi_{c(A_j)}^{s_j}\bigg)\bigg(\prod_{j=1}^n \Psi_{A_j}\bigg)S_{\ve{w}',\ve{z}'}^+&& \text{(Lemma \ref{lem:wendpointcommutator})}\\
&=S_{\ve{w}',\ve{z}'}^- \sum_{s\in \{0,1\}^n}\bigg(\prod_{j=1}^n \Phi_{w_j}^{s_j}\bigg)\bigg(\prod_{j=1}^n \Psi_{K}^{s_j}\bigg)\bigg(\prod_{j=1}^n \Psi_{A_j}\bigg)S_{\ve{w}',\ve{z}'}^+&& \text{(Lemmas \ref{lem:Psi_Ascommute}, \ref{lem:psi_Kpsi_A})}\\
&=\sum_{s\in \{0,1\}^n}\bigg(\prod_{j=1}^n \Phi_{w_j}^{s_j}\bigg)\bigg(\prod_{j=1}^n \Psi_{K}^{s_j}\bigg)S_{\ve{w}',\ve{z}'}^-\bigg(\prod_{j=1}^n \Psi_{A_j}\bigg)S_{\ve{w}',\ve{z}'}^+ &&\text{(Lemmas \ref{lem:quasistabcommuteswithpsi}, \ref{Spm-Psicommutatornotadjacent})}\\
&=\sum_{s\in \{0,1\}^n}\bigg(\prod_{j=1}^n \Phi_{w_j}^{s_j}\bigg)\bigg(\prod_{j=1}^n \Psi_{K}^{s_j}\bigg)&&\text{(Equation \eqref{eq:manytrivialstrands})}.\end{align*}

We now note that by Lemma \ref{lem:Psi_A^2=0}, if $s\in \{0,1\}^n$ then

\[\bigg(\prod_{j=1}^n \Psi_{K}^{s_j}\bigg)\simeq 0\] if $s_j$ is nonzero for more than one $j$. Hence the above sum reduces to
\[\varsigma_*\simeq 1+\sum_{j=1}^n \Phi_{w_i}\Psi_K=1+\Phi_K\Psi_K,\] completing the proof.
 \end{proof}

 \subsection{The map associated to a partial twist around a link component}
 
 In this section, we perform an additional basepoint moving map computation and prove Theorem \ref{thm:D}. Suppose that $\bL$ is a multibased link and $K$ is a component with basepoints $z_1,w_1,z_2,w_2,\dots, z_n, $ and $w_n$, appearing in that order. Let $\tau$ be the diffeomorphism induced by twisting $(\frac{1}{n})^{\textrm{th}}$ of the way around $K$, sending $z_i$ to $z_{i+1}$ and $w_i$ to $w_{i+1}$ (with indices taken modulo $n$). In the case that we pick a coloring $(\sigma, \frP)$ where all of the $\ve{w}$-basepoints have the same color, the map $\tau$ induces a map on the complex
 \[CFL_{UV}^\infty(Y,\bL,\sigma, \frP,\frs).\] We have the following:
 
 \begin{customthm}{\ref{thm:D}}Suppose that $\bL$ is an embedded multibased link in $Y$, and $K$ is a component of $\bL$ with basepoints $z_1,w_1,\dots, z_n,$ and $w_n$, appearing in that order. Assume that $n>1$. If $\tau$ denotes the $(\frac{1}{n})^{\textrm{th}}$-twist map, then for a coloring where all $\ve{w}$-basepoints on $K$ have the same color, we have
 \[\tau_*\simeq (\Psi_{z_1}\Phi_{w_1}\Psi_{z_2} \Phi_{w_2} \cdots \Phi_{w_{n-1}} \Psi_{z_n} \Phi_{w_n} )
 +(\Phi_{w_1}\Psi_{z_2} \Phi_{w_2} \cdots \Phi_{w_{n-1}} \Psi_{z_n}).\]
 \end{customthm}
 
 \begin{proof}Let $A_i$ be the arc from $w_i$ to $w_{i+1}$, respecting the orientation of $K$. Let $w'$ and $z'$ be new basepoints in the region between $z_n$ and $w_1$. Let $A'$ denote the arc from $w_n$ to $w'$ and let $A''$ denote the arc from $w'$ to $w_1$. This is illustrated in Figure \ref{fig::31}. 
 
  \begin{figure}[ht!]
\centering
\includegraphics[scale=1.2]{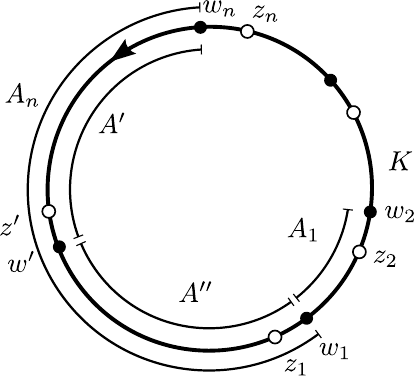}
\caption{The basepoints $z_1,w_1,\dots, z_n,w_n,z',w'$ as well as the arcs $A_i,A'$ and $A''$ from Theorem \ref{thm:D}. \label{fig::31}}
\end{figure}

Using Lemma \ref{lem:movebasepoints} repeatedly, we have that
\[\tau_*\simeq (S_{w'}^-\Psi_{A''}S_{w_1}^+)(S_{w_1}^- \Psi_{A_1} S_{w_2}^+)\cdots ( S_{w_{n-1}}^- \Psi_{A_{n-1}} S_{w_n}^+)(S_{w_n}^- \Psi_{A'} S_{w'}^+).\] Using this, we perform the following computation:
\begin{align*}\tau_* &\simeq S_{w'}^-\Psi_{A''}\Phi_{w_1}\Psi_{A_1} \Phi_{w_2} \cdots \Phi_{w_{n-1}} \Psi_{A_{n-1}} \Phi_{w_n} \Psi_{A'} S_{w'}^+&&\text{(Lemma \ref{lem:del'appears})}\\
&\simeq S_{w'}^-\Psi_{A''}\Phi_{w_1}\Psi_{A_1} \Phi_{w_2} \cdots \Phi_{w_{n-1}} \Psi_{A_{n-1}} (\Psi_{A'}\Phi_{w_n}+1)  S_{w'}^+&& \text{(Lemma \ref{lem:wendpointcommutator})}\\
&\simeq S_{w'}^-\Psi_{A''}\Phi_{w_1}\Psi_{A_1} \Phi_{w_2} \cdots \Phi_{w_{n-1}} \Psi_{A_{n-1}} \Psi_{A'}\Phi_{w_n}  S_{w'}^+
&&\\
&\qquad +S_{w'}^-\Psi_{A''}\Phi_{w_1}\Psi_{A_1} \Phi_{w_2} \cdots \Phi_{w_{n-1}} \Psi_{A_{n-1}}  S_{w'}^+&&\\
&\simeq S_{w'}^-(\Psi_{A''}\Psi_{A'})\Phi_{w_1}\Psi_{A_1} \Phi_{w_2} \cdots \Phi_{w_{n-1}} \Psi_{A_{n-1}} \Phi_{w_n}  S_{w'}^+&&\text{(Lemmas \ref{lem:wendpointcommutator}, \ref{lem:Psi_Ascommute})}\\
&\qquad +S_{w'}^-\Psi_{A''}\Phi_{w_1}\Psi_{A_1} \Phi_{w_2} \cdots \Phi_{w_{n-1}} \Psi_{A_{n-1}}  S_{w'}^+&& \\
& \simeq S_{w'}^-(\Psi_{A''}\Psi_{A_n})\Phi_{w_1}\Psi_{A_1} \Phi_{w_2} \cdots \Phi_{w_{n-1}} \Psi_{A_{n-1}} \Phi_{w_n}  S_{w'}^+&&\text{(Lemma \ref{lem:Psi_A^2=0})}\\
&\qquad +S_{w'}^-\Psi_{A''}\Phi_{w_1}\Psi_{A_1} \Phi_{w_2} \cdots \Phi_{w_{n-1}} \Psi_{A_{n-1}}  S_{w'}^+&& \\
&\simeq (S_{w'}^-\Psi_{A''}S_{w'}^+)\Psi_{A_n}\Phi_{w_1}\Psi_{A_1} \Phi_{w_2} \cdots \Phi_{w_{n-1}} \Psi_{A_{n-1}} \Phi_{w_n}  &&\text{(Lemma \ref{lem:quasistabcommuteswithpsi})}\\
&\qquad +(S_{w'}^-\Psi_{A''}S_{w'}^+)\Phi_{w_1}\Psi_{A_1} \Phi_{w_2} \cdots \Phi_{w_{n-1}} \Psi_{A_{n-1}}  && \\
&\simeq \Psi_{A_n}\Phi_{w_1}\Psi_{A_1} \Phi_{w_2} \cdots \Phi_{w_{n-1}} \Psi_{A_{n-1}} \Phi_{w_n}  &&\text{(Lemma \ref{lem:addtrivialstrand})}\\
&\qquad +\Phi_{w_1}\Psi_{A_1} \Phi_{w_2} \cdots \Phi_{w_{n-1}} \Psi_{A_{n-1}}  && 
\end{align*}
 completing the proof since by definition $\Psi_{A_i}=\Psi_{z_{i+1}}$ on the complex $CFL^{\infty}_{UV}(Y, \bL,\sigma,\frP,\frs)$. \end{proof}

\end{document}